\documentclass[a4paper,14pt]{article}
\usepackage[utf8]{inputenc}
\usepackage{float}
\usepackage{cite}
\usepackage[russian,english]{babel}
\usepackage{amssymb,amsmath,amsfonts}
\usepackage{color}
\usepackage{enumerate}
\usepackage[dvips]{graphicx}
\usepackage{setspace}
\usepackage{xcolor}
\usepackage{fancyhdr}
\usepackage{amsmath}
\usepackage[all,matrix,arrow,curve]{xy}

\newcommand{\RomanNumeralCaps}[1]
  {\MakeUppercase{\romannumeral #1}}
\DeclareMathOperator{\tr}{\mathrm{tr}}
\DeclareMathOperator{\F}{\mathrm{\mathbb{F}}}

\newtheorem{thm}{Theorem}
\newtheorem{proof}{Proof}
\newtheorem{definition}{Definition} 
\newtheorem{lemma}{Lemma} 
\newtheorem{theorem}{Theorem} 
\newtheorem{remark}{Remark} 

\newtheorem{lem}{Lemma}
\newtheorem{prop}{Proposition}
\newtheorem{cor}{Corrollary}

\newtheorem{exm}{Example}

\newtheorem{rem}{Remark}

\renewcommand{\leq}{\leqslant}
\renewcommand{\geq}{\geqslant}
\renewcommand{\phi}{\varphi}

\renewcommand{\l}{\bigl\langle} \renewcommand{\r}{\bigr\rangle}

\begin{document}
\pagestyle{fancy}
\fancyhead{}
\fancyhead[L]{\footnotesize{Analytical formulas of roots in $ESL_3(F_p)$}}
\fancyhead[R]{\footnotesize{Ruslan Skuratovskii}}
\fancyfoot{}
\fancyfoot[L]{\footnotesize{Extended  Special Linear group $ESL_3(\mathbb{Z})$}}
\fancyfoot[R]{\footnotesize{ June, 2025}}
\renewcommand{\footrulewidth}{0.1 mm}

\begin{center}
\textbf{Involutive minimal generating sets of Extended  Special Linear group $ESL_5(\mathbb{Z})$, $ESL_5(F_p)$, formulas of roots in $ES{{L}_{2}}({{\mathbb{F}}_{p}})$, $ES{{L}_{2}}({{\mathbb{Z}}})$ and $S{{L}_{3}}({{\mathbb{Z}}})$ \RomanNumeralCaps{2}.
}
\vspace{\baselineskip}

\author{R. Skuratovskii\\
{ruslan.skuratovskii@imath.kiev.ua} \\
Institute of applied mathematics and mechanics \\
of the NASU, Kyiv, Ukraine, skuratovskii@nas.gov.ua}

{Skuratovskii  R.V. \\
{ruslan.skuratovskii@imath.kiev.ua} \\
Institute of applied mathematics and mechanics \\
of the NASU, Kiev, Ukraine, skuratovskii@nas.gov.ua}
\end{center}

{\bf Abstract }
The size of minimal systems of generators and these systems themselves for groups $ESL_n \left[ \mathbb{Z} \right]$ and $ESL_n \left( \mathbb{F}_p \right)$ were found.
 Moreover we obtain triples of involutions with\textit{\textbf{ two commuting involu\-tions}} generating $ESL_5 \left[ \mathbb{Z} \right]$, $ESL_5 \left[ \mathbb{Z} \right]$ relatively,  that are Mazurov triples.

An analytical formula of root in $SL(3, \mathbb{Z})$ is found, the analytical formula for $4$-th power root in $SL(2, \mathbb{Z}$) is found too. The $n$-th power root of elements of $GL(2, \mathbb{Z})$ is found in the form of recursive formula. 

The roots existing problem for different classes of matrix --- simple and semisimple matrixes from $SL_2({\mathbb{F}})$, $ SL_2({\mathbb{Z}})$ and $ GL_2({\mathbb{F}})$ are solved.

We introduce the concept of extended special linear group $ESL_{2k+1}(\mathbb{F})$, which is generalization of the matrix group $SL_{2k+1}(\mathbb{F})$, where $\mathbb{F}$ is arbitrary perfect field.
Such matrix groups are widely used in non-commutative cryptography, as well as in the key exchange protocol \cite{SkuKeyExchange} and in post-quantum cryptography, see \cite{NTRU}.


\section{Introduction}
In this research we continue our previous investigation \cite{SkuESL}. We generalize the group of unimodular matrices \cite{Amit, Un, AutofUnimod} and find its structure. For this goal we propose one extension of the special linear group. Groups generated by three involutions, two of which are permutable, have long been of interest in the theory of matrix groups \cite{Maz}, for instance such generating set was researched for $S{{L}_{2}}({{\mathbb{Z}+ i\mathbb{Z}}})$. But for size of matrix 3 on 3 this is impossible for some groups. We research this question for $ES{{L}_{2}}({{\mathbb{Z}}})$.

The question of study an automorphism group of lattice $\mathbb{Z}^n$ is very important because it is an algebraic foundations of lattice-based cryptography \cite{LatCry, NTRU, BookLatCry}. 
Such matrix groups are widely used in non-commutative cryptography, as well as in the key exchange protocol \cite{SkuKeyExchange} and in post-quantum cryptography, see \cite{NTRU}.

Group of automorphisms of $\mathbb{Z}^n$ is much bigger than just signed permutations, it's infinite (for $n \geq 2$) and consists of all integer matrices whose determinant is 1 or -1.
An automorphism of this lattice be an orthogonal transformation (i.e. linear transformation that preserves the inner product).

The automorphism group of $\mathbb{Z}^n$ is not finite when $n \geq 2$, since it consists of all $n\times n$ matrices with integer coefficients and determinant 1 or -1 that means this automorphism group is $ES{{L}_{n}}[{{\mathbb{Z}}}]$. 

In the Diffie–Hellman protocol, public parameters, private secret values, public keys, and a final shared secret key are distinguished. In this protocol public parameters are elements of a commutative group. In the new key exchange protocol \cite{SkuKeyExchange} which is generalization of Diffie–Hellman protocol, these will be elements of a commutative subgroup or a commutative subset of elements of a non-commutative group. This is why it is important to find commutative involutions in the group. Note that the fastest computations are possible with elements of order 2.

An analytical formula of root in $SL(3, \mathbb{Z}$) is found, recursive formula for $n$-th power root in $SL(2, \mathbb{Z}$) is found too.

We denote iff --- necessary and sufficient condition, 
e.v.  --- eigenvalue.
Let $\mu_A$ be minimal polynomial of $A$.

Recall that matrix $A$ is
called semisimple if $\mu_A$ is a product of distinct monic irreducible and
separable polynomials; if moreover all these irreducible polynomials
have degree 1, then $A$ is called split semisimple or diagonalizable. The \textit{trace} of a square matrix is the sum of its diagonal entries.

\section{Concept of $ESL_3(\mathbb{F}_p)$ and $ESL_3(\mathbb{Z})$}
We recall concept of $ES{{L}_{2}}\left( \mathbb{F}_p \right)$ introduced in our previous work \cite{SkuESL}.
\begin{definition}
 The set of matrices
\begin{equation}\label{ESL}
 \left\{ {{M}_{i}}:Det({{M}_{i}})=\pm 1,  {M}_{i} \in GL_2(\mathbb{F}_p) \right\} 
 \end{equation}
forms extended special linear group and is denoted by ${ESL}_2(\mathbb{F}_p)$.
\end{definition}

Let $SL_3(\mathbb{F}_p)$  denotes the special linear group of degree 3 over finite field $\mathbb{F}_p$.

\textbf{\bf Definition 1.}
\textit{ The set of matrices \[\left\{ {{M}_{i}}:Det({{M}_{i}})=\pm 1,  {M}_{i} \in GL_3(\mathbb{F}_p) \right\} \] forms \textbf{extended special linear group}  in $GL_3(\mathbb{F}_p)$ and is denoted by  ${ESL}_3(\mathbb{F}_p)$.}

By the transvection $t_{ij}$ we mean the sum $E+e_{ij}$, where $e_{ij}$ is a matrix unit with 1 only in  intersection of $i$-th row and $j$-th column the rest elements are 0.   For instance     $t_{12}=\left( \begin{matrix}
   1 & 1 & 0  \\
   0 & 1 & 0  \\
   0 & 0 & 1  \\
\end{matrix} \right)$.

Denote a permutation matrix of order 3 by ${{P}_{3}}$ and the transvection \cite{Klyach} by $tr_{12}$ of group $SL_3(\mathbb{F}_p)$.
Suppose ${{D}_{123}}=\left( \begin{matrix}
   -1 & 0 & 0  \\
   0 & -1 & 0  \\
   0 & 0 & -1  \\
\end{matrix} \right)$ be additional generator extending $SL\left( 3, \mathbb{F}_p \right)$ to $ESL\left( 3,\mathbb{F}_p \right)$.


\textbf{\bf Definition 2.}
\textit{ The set of matrices \[\left\{ {{M}_{i}}:Det({{M}_{i}})=\pm 1,  {M}_{i} \in GL_3(\mathbb{F}_p) \right\} \] forms \textbf{extended special linear group} over $\mathbb{F}_p$ and is denoted by ${ESL}_3(\mathbb{F}_p)$.}

\textbf{\bf Definition 3.}
\textit{ The set of matrices \[\left\{ {{M}_{i}}:Det({{M}_{i}})=\pm 1,  {M}_{i} \in GL_3[\mathbb{Z}] \right\} \] forms \textbf{extended special linear group} over $\mathbb{Z}$ and is denoted by  ${ESL}_3[\mathbb{Z}]$.}

To construct split extension of $SL_3\left( \mathbb{F}_p \right)$ we consider the following class of matrices
with determinant $-1$,
but do not centralizing the group $S{{L}_{3}}\left( \mathbb{F}_p \right)$ and one exception that's diagonal matrix.
$${{D}_{1}}=\left( \begin{matrix}
   -1 & 0 & 0  \\
   0 & 1 & 0  \\
   0 & 0 & 1  \\
\end{matrix} \right), 
 {{D}_{2}}=\left( \begin{matrix}
   1 & 0 & 0  \\
   0 & -1 & 0 \\
   0 & 0 & 1  \\
\end{matrix} \right),
         {{D}_{3}}=\left( \begin{matrix}
   1 & 0 & 0  \\
   0 & 1 & 0  \\
   0 & 0 & -1  \\
\end{matrix} \right).$$

Based on the above, we conclude that structure of group generated as extension of $S{{L}_{3}}\left( \mathbb{F}_p \right)$ by $\left\langle {{D}_{1}} \right\rangle $ is semidirect product 
$$\left\langle {{D_{1}}} \right\rangle \ltimes SL_{3}\left( \mathbb{F}_p \right)\simeq ES{{L}_{3}}\left( \mathbb{F}_p \right)$$
with kernel $S{{L}_{3}}\left( \mathbb{F}_p \right)$.

The role of a complement subgroup for the kernel $SL_3(\mathbb{F})$ in the split group $ESL_3(\mathbb{F})$ can also be played by subgroups formed by $D_2$ and $D_3$ therefore $\left\langle {{D_{2}}} \right\rangle \ltimes SL_{3}\left( \mathbb{F}_p \right)\simeq ES{{L}_{3}}\left( \mathbb{F}_p \right)$ and $\left\langle {{D_{3}}} \right\rangle \ltimes SL_{3}\left( \mathbb{F}_p \right)\simeq ES{{L}_{3}}\left( \mathbb{F}_p \right)$.

A diagonal matrix in which one diagonal element is equal to $-1$ and the rest are equal to 1 is called an \textit{\textbf{elementary diagonal}} matrix.

But if we use the diagonal matrix $D_{1,2,3}=D_1D_2 D_3$, which centralizes $SL_3 (\mathbb{F}_p)$ as a scalar matrix, then the semidirect product degenerates into a direct product
$$\langle D_{1,2,3} \rangle \times SL_3 (\mathbb{F}_p) \simeq ESL_3 (\mathbb{F}_p).$$

Let $D$ be a subgroup generated by elementary diagonal matrices from $SL(n,p)$ and $D^{-}$ be its adjacency class whose matrices $D_i$ have negative determinants.

\textbf{Remark.} {\sl 
If $n=2k+1$, then there exist exactly $2^{n-1}$ diagonal extensions by subgroup generated by a matrix $ d_i \in D^{-}$ of the group $SL(n, \mathbb{F}_p)$ to the group $ESL(n, \mathbb{F}_p)$.

This extension takes the form $\left\langle {{D_{i}}} \right\rangle \ltimes SL_{3}\left( \mathbb{F}_p \right)$ in the case $D_i$ is not a scalar matrix and form the trivial central extension $\langle D_i \rangle \times SL_3 (\mathbb{F}_p) \simeq ESL_3 (\mathbb{F}_p)$ if $D_i$ is a scalar matrix.}


\textbf{Proof.}
To generate diagonal matrix with odd  number of -1 by using elementary diagonal matrices $D_i$ we can multiplicate odd number of them. The quantity of diagonal matrix with odd (even) number of -1 on diagonal is sum of combinations of $n$ by odd numbers means coordinates of -1 on diagonal.
In order to compute a sum of ${n^{} \choose 2l}$ consider a sum of $(1+1)^n=2^n$ and $(1-1)^n=0$ which is $ \left({n^{} \choose 0}+{n^{} \choose 2}+ {n^{} \choose 4}+ \ldots \right)$ and divide it by 2.
 
\begin{equation}\label{placement}
\left(\sum_{l=0}^{{n^{}}}{n^{} \choose 2l}(1+(-1)^l)\right):2 =
2^{n^{}-1}.    
\end{equation}
To calculate the sum of odd combinations, we subtract the sum of even combinations of $n$ elements that is $2^{n^{}-1}$ from $2^n$.

By the transvection $t_{ij}$ we mean the sum $E+e_{ij}$, where $e_{ij}$ is a matrix unit with 1 only in  intersection of $i$-th row and $j$-th column the rest elements are 0.   

Denote a permutation matrix of order 3 by ${{P}_{3}}$ and the transvection \cite{SkuESL} by $tr_{12}$ of group $SL_3(\mathbb{F}_p)$.
Suppose ${{D}_{1,2,3}}=\left( \begin{matrix}
   -1 & 0 & 0  \\
   0 & -1 & 0  \\
   0 & 0 & -1  \\
\end{matrix} \right)$ be additional generator extending $SL_3\left( , \mathbb{F}_p \right)$ to $ESL_3 \left( \mathbb{F}_p \right)$.

\textbf{Proposition 1.}
{\sl  Minimal generating set for $ESL_3(\mathbb{Z})$ consists of 2 generators: 
\begin{center}
  $P=\left( \begin{matrix}
   0 & -1 & 0  \\
   0 & 0 & 1  \\
   1 & 0 & 0  \\
\end{matrix} \right)$ and  $t_{12}=\left( \begin{matrix}
   1 & 1 & 0  \\
   0 & 1 & 0  \\
   0 & 0 & 1  \\
\end{matrix} \right).$ 
\end{center} }
The order of the permutation matrix is indicated by the relation $P^5=E$.
The size of the generic set is minimal for non-cyclic group so its minimality does not order a proof.
The order of generated  group is 11232 is in 2 times grater then order of $SL_3(\mathbb{Z})$.

For convenience we fix some notations for \textit{diagonal involutive matrices} from $ESL_3(\mathbb{Z})$:
\begin{center}  
$I_{12} = \begin{pmatrix}
            -1 & 1 & 0 & 0 & 0 \\
            0  & 1 & 0 & 0 & 0 \\
            0  & 0 & 1 & 0 & 0 \\
            0  & 0 & 0 & 1 & 0 \\
            0  & 0 & 0 & 0 & 1
        \end{pmatrix},
        I_{23} = \begin{pmatrix}
            1 & 0 & 0 & 0 & 0 \\
            0 & -1& 1 & 0 & 0 \\
            0 & 0 & 1 & 0 & 0 \\
            0 & 0 & 0 & 1 & 0 \\
            0 & 0 & 0 & 0 & 1
        \end{pmatrix}, \ldots,
        I_{45} = \begin{pmatrix}
            1 & 0 & 0 & 0 & 0 \\
            0 & 1 & 0 & 0 & 0 \\
            0 & 0 & 1 & 0 & 0 \\
            0 & 0 & 0 & -1& 1 \\
            0 & 0 & 0 & 0 & 1
        \end{pmatrix}.
    $
    \end{center}

By the transvection $t_{ij}$ we mean the sum $E+e_{ij}$, where $e_{ij}$ is a matrix unit with 1 only in  intersection of $i$-th row and $j$-th column the rest elements are 0.   For instance     $t_{12}=\left( \begin{matrix}
   1 & 1 & 0  \\
   0 & 1 & 0  \\
   0 & 0 & 1  \\
\end{matrix} \right)$.

Denote a permutation matrix of order 3 by ${{P}_{3}}$ and the transvection \cite{Klyach} by $tr_{12}$ of group $SL_3(\mathbb{Z})$.
Suppose ${{D}_{123}}=\left( \begin{matrix}
   -1 & 0 & 0  \\
   0 & -1 & 0  \\
   0 & 0 & -1  \\
\end{matrix} \right)$ be additional generator extending $SL\left( 3, \mathbb{Z} \right)$ to $ESL\left( 3,\mathbb{Z} \right)$.

Based on the above, we conclude that structure of group generated as extension of $S{{L}_{3}}\left( \mathbb{Z} \right)$ by $\left\langle {{D}_{1}} \right\rangle $ is semidirect product 
$$\left\langle {{D_{1}}} \right\rangle \ltimes SL_{3}\left( \mathbb{Z} \right)\simeq ES{{L}_{3}}\left( \mathbb{Z} \right)$$
with kernel $S{{L}_{3}}\left( \mathbb{Z} \right)$.

This group admits such generating sets $\left<{{D}_{1}},\,\,t{{r}_{12}},\,\,t{{r}_{32}},\,\,{{P}_{3}}\right>$, $\left<{{D}_{2}},\,\,t{{r}_{12}},\,\,t{{r}_{32}},\,\,{{P}_{3}} \right>$ and $\left<{{D}_{3}},\,\,t{{r}_{12}},\,\,t{{r}_{32}},\,\,{{P}_{3}}\right>$.

If we substitute generator $D_{123}$ instead of $D_{1}$ then in terms of these subgroups (factors) a decomposition in product takes form $$\left\langle {{D_{123}}} \right\rangle \times S{{L}_{3}}\left( \mathbb{Z} \right)\simeq ES{{L}_{3}}\left( \mathbb{Z} \right),$$ because $D_{123}$ centralize $S{{L}_{3}}\left( \mathbb{Z} \right)$.

As it is studied by us $ESL_3(\mathbb{Z})$ has a structure of semidirect product $SL_3(\mathbb{Z}) \rtimes \,{\left\langle\mathbb{{D}}_{1} \right\rangle }$.

The existence of a non-trivial homomorphism $\varphi :\,\,{{\mathbb{Z}}_{2}}\to Aut\left( S{{L}_{2}}(\mathbb{Z}) \right)$, as well as $\phi :\,\,{{\mathbb{Z}}_{2}}\to Aut\left( S{{L}_{2}}({{\mathbb{F}}_{p}}) \right)$ can be proved by indicating an element of order 2 in the automorphisms of base group that is the kernel of the semidirect product we want to construct.
There is countergradient automorphism in $S{{L}_{3}}\left( \mathbb{Z} \right)$ that is $\varphi :\,M\to {{\left( {{M}^{T}} \right)}^{-1}}$ or alternating automorphism of order 2 acting by conjugating $\varphi :\,M\to D_1^{-1}MD_1$, where $D_1=\left( \begin{array}{rrr}
   -1\,\,\, &0\,\,\, &0 \\
   0\,\,\, &1 \,\,\, &0  \\
  0\,\, & 0  \,\,\, &1\\
\end{array} \right)$ and is called by diagonal automorphism \cite{Mersl}.

Recall the \textbf{definition} of $\mathbf{TI-subgroup}$ \cite{Suds,Zu}.  Let $G$ be a group and $A < G$, then $A$ is called $\mathbf{TI-}$subgroup iff  $A \cap A^g = e$ for each $g \in G \setminus N_G(A)$.

\begin{rem}
Subgroup $\mathbb{C}_{2}$ is $\mathbf{TI-subgroup}$ and not antinormal subgroup.
\end{rem}

\begin{proof}
In view of ${\mathbb{C}}_{2}$ is one generated then its centralizer coincides with its normalizer. One easy can verify that centralizer consists of all diagonal matrices from $ESL_2(\mathbb{F}_p)$.
Let us find a structure of such normalizer $N_{ESL_2(\mathbb{F}_p)} ({\mathbb{C}}_{2})$.
In view of e.v. is  invariant under conjugation by non-singular matrix over field the normalizer of top subgroup ${\mathbb{C}}_{2}$ in $ESL_2(\mathbb{F}_p)$ consists of all diagonal matrices from $ESL_2(\mathbb{F}_p)$ and permutational matrix ${{\mathcal{P}}}=\left( \begin{array}{rr}
   0\,\,\, &1 \\
  1\,\, & 0 \\
\end{array} \right)$.
We assume that $N_{ESL_2(\mathbb{F}_p)} ({\mathbb{C}}_{2}) \simeq   D(SL_2(\mathbb{F}_p))\rtimes \,{\mathcal{P}}$, where $D(SL_2(\mathbb{F}_p))$ diagonal subgroup of $ESL_2(\mathbb{F}_p)$.

For the rest of elements condition of $A \cap A^g = e$ for each $g \in ESL_2(\mathbb{F}_p) \setminus N_{ESL_2(\mathbb{F}_p)} ({\mathcal{C}}_{2})$ holds. Thus, $\mathbb{C}_{2}$ is $\mathbf{TI-subgroup}$, hence $\mathbb{C}_{2}$ but is antinormal subgroup because of its normalizer $N_{ESL_2(\mathbb{F}_p)} ({\mathcal{C}}_{2})$ do not coincides with $({\mathcal{D}}_{i})$.

\end{proof}
 Assume that $\mathbb{D}_{i} \in \{D_1, D_2, \ldots , D_n \} \setminus D_{12...n}$.
\begin{rem}
If the complement subgroup to kernel $SL_2(\mathbb{F}_p)$ of semidirect product is $\langle \mathbb{D}_{i} \rangle$, then $\mathbb{D}_{i}$ is $\mathbf{TI-subgroup}$ but not antinormal subgroup. The subgroup $D_{12...n}$ is not $\mathbf{TI-subgroup}$. 
\end{rem}

\begin{proof}
In view of ${\mathbb{D}}_{i}$ is one generated then its centralizer coincides with its normalizer. One easy can verify that centralizer consists of all diagonal matrices from $ESL_2(\mathbb{F}_p)$.
Let us find a structure of such normalizer $N_{ESL_2(\mathbb{F}_p)} ({\mathbb{C}}_{2})$.
In view of e.v. is  invariant under conjugation by non-singular matrix over field the normalizer of top subgroup ${\mathbb{D}}_{i}$ in $ESL_3(\mathbb{F}_p)$ consists of  all diagonal matrices from $ESL_2(\mathbb{F}_p)$ and permutational matrix ${{\mathcal{P}}}=\left( \begin{array}{rr}
   0\,\,\, &1 \\
  1\,\, & 0 \\
\end{array} \right)$.
We assume that $N_{ESL_2(\mathbb{F}_p)} ({\mathbb{C}}_{2}) \simeq   D(SL_2(\mathbb{F}_p))\rtimes \,{\mathcal{P}}$, where $D(SL_2(\mathbb{F}_p))$ diagonal subgroup of $ESL_2(\mathbb{F}_p)$.

For the rest of elements condition of $A \cap A^g = e$ for each $g \in ESL_2(\mathbb{F}_p) \setminus N_{ESL_2(\mathbb{F}_p)} ({\mathcal{C}}_{2})$ holds. Thus, $\mathbb{D}_{i}$ is $\mathbf{TI-subgroup}$, but in view of ${{N}_{ESL_3(\mathbb{Z})}}\left( {{D}_{1}} \right) \neq {D}_{1}$ it imply that $\mathbb{D}_{i}$ is not antinormal subgroup 


Thus, $\mathbb{D}_{i}$ is a $\mathbf{TI-subgroup}$, but the inequality ${{N}_{ESL_3(\mathbb{Z})}}\left( {{D}_{i}} \right) \neq {D}_{i}$ immediately entails that $\mathbb{D}_{i}$ is not an antinormal subgroup.

Consider $D_{12...n} \cap D_{12...n}^g$ for each $g \in ESL_n(\mathbb{F}_p) \setminus N_{ESL_n(\mathbb{F}_p)} ({\mathcal{D}}_{12...n})$. Since $N_{ESL_n(\mathbb{F}_p)} ({\mathcal{D}}_{12...n})=ESL_n(\mathbb{F}_p)$, then $ESL_n(\mathbb{F}_p) \setminus N_{ESL_n(\mathbb{F}_p)} ({\mathcal{D}}_{12...n}) = \empty $, thence $D_{12...n} \cap D_{12...n}^g= D_{12...n}$.That completes the proof.
\end{proof}
Additionally we remark that as a consequence $\left\langle {D_{12...n}} \right\rangle$ is normal.

Let us find a normal closure of $D_i$ for $1 \leq i \leq 3$ in $ESL(3,Z)$ which be denoted by ${{N}_{ESL}}\left( {{D}_{1}} \right)$. We demonstrate two typical classes of ${{N}_{ESL_3(\mathbb{Z})}}\left( {{D}_{i}} \right)$ normal closure here: 
\[{{N}_{ESL_3(\mathbb{Z})}}\left( {{D}_{1}} \right)=\left\{ \left( \begin{matrix}
   d & 0 & 0  \\
   0 & a & b  \\
   0 & g & c  \\
\end{matrix} \right) \in ESL_3(\mathbb{Z}) \left| d,a,b,c,g\in Z \right. \right\}\]
\[{{N}_{ESL}}\left( {{D}_{2}} \right)=\left\{ \left( \begin{matrix}
   x & 0 & f  \\
   0 & y & 0  \\
   c & 0 & z  \\
\end{matrix} \right)\in ESL_3(\mathbb{Z}) \left| x,y,f,c,z\in Z \right. \right\}. \]







The intersection $A \cap A^g = E$, provided $g \notin {{N}_{ESL_3(\mathbb{Z})}}\left( {{D}_{i}} \right)$ for each $1 \leq i \leq 3$, is trivial by virtue of ${{N}_{ESL_3(\mathbb{Z})}}\left( {{D}_{i}} \right)={C}_{{ESL}_3(\mathbb{Z})}\left( {{D}_{i}} \right)$ (a normalizer of a one-generated group is equal to its centralizer).

The exceptional case of complement subgroup for the kernel $SL_3(\mathbb{Z})$. The normalizer $N_{ESL_3(\mathbb{Z})}({\langle {D}_{123}}\rangle)$ of the subgroup $\langle D_{123} \rangle$ is an exception as an element of the centre of the entire group $ESL_3(\mathbb{Z})$, therefore ${{N}_{ESL_3(\mathbb{Z})}}\left( \left\langle {{D}_{123}} \right\rangle \right) = ESL_3(\mathbb{Z})$ and consequently $\left\langle {D_{123}} \right\rangle$ is normal. Hence, this subgroup is not antinormal and not TI. 

\begin{definition} \label{quasimple}
A group $G$ is called quasisimple if its inner automorphism group $Inn(G)$ is simple or, equivalently, ${}^G/{}_{Z{(G)}}$ is simple.
\end{definition}

\begin{definition}
      We define $PESL(n,p)$ as a quotient of $ESL(n,p)$ by its center.
\end{definition}

\begin{rem}\label{SLn=2k+1}
In the general case $n=2k+1$, there is an isomorphism $SL(2k+1,\mathbb{Z}) \simeq PSL(2k+1,\mathbb{Z})$ due to the absence of non-trivial scalar matrices $A_i$ with $\det(A_i)=-1$   
\end{rem}
\begin{proof}
    The kernel of homomorphism from $SL(3, \mathbb{Z})$ to $PSL(3,\mathbb{Z})$ consists only of scalar matrix $E$ due to the absence of 
another scalar matrices $A_i$ with $\det(A_i)=1$ over $\mathbb{Z}$.  Therefore $SL(3,\mathbb{Z}) \simeq PSL(3,\mathbb{Z})$. 
\end{proof}

Let $i$ denotes embedding of $SL(3,\mathbb{Z})$ in $ESL(3,\mathbb{Z})$.
 Thus we have quotient 
 ${}^{ESL(3,Z)}/{}_{ <E, -E>} \simeq PESL(3,\mathbb{Z})$. 

\begin{center}
\textbf{Commuting diagram for $n=2k+1$ is similar for the case $n= 3$}
\end{center}
\begin{equation}\label{diag:3,Z}
\xymatrix{
SL(n, \mathbb{Z}) \ar @{^{(}->}@<-1ex>^{\phi} [r] \ar@{.>}[dr]|-{(i)}   \ar @{<->}[d]_{\simeq}
& ESL(n,\mathbb{Z}) \ar @{->>}@<-1ex>[l]_{\mathfrak{C}} \ar @{->>}  [d] ^{\psi}  \\
PSL(n,\mathbb{Z})                        & PESL(n, \mathbb{Z}) \ar @{<<->>}[l]_{=} }
\end{equation}  

\begin{thm}\label{n=2k+1}
For $n=2k+1$ we have $PESL\left( n,p \right)=PSL\left( n,p \right)$ an index of a center 
$\left[ Z\left( ESL\left( 2k+1,p \right) \right): \right. \\ \left.  Z\left( SL\left( 2k+1,p \right) \right) \right]=2$. Furthermore, $PESL\left( n, p\right)$ is simple, except for special cases of $( n, p )$.   
\end{thm}
\begin{proof}
Taking into account that equation ${{a}^{2k+1}}=-1$ has non-trivial solutions in every ${{\text{F}}_{p}}$, $p>1$, then the number of scalar matrices with $\det(A)=-1$ is equal to the number of scalar matrices with $\det(A)=1$. 
Then the number of scalar matrices with $\det(A)=-1$ coincides with the number of scalar matrices with $\det(A)=1$, therefore 
$\left[  ESL\left( 2k+1,p \right): \\  SL\left( 2k+1,p \right) \right]=2$.  This determines the center index $\left[ Z\left( ESL\left( 2k+1,p \right) \right):Z\left( SL\left( 2k+1,p \right) \right) \right]=2$.
Therefore, $PESL\left( n,p \right)= PSL\left( n,p \right)$ in this case.
Thus, according to Definition \ref{quasimple} in the case $PSL\left( n,p \right)$ is simple $ESL\left( n,p \right)$ is quasisimple. 
\end{proof}
Here the equality $PESL(2k+1,\mathbb{Z}) = PSL(2k+1,\mathbb{Z})$ is provided via $\left[ Z\left( ESL\left( 2k+1,p \right) \right):Z\left( SL\left( 2k+1,p \right) \right) \right]=2$ stated in the Theorem \ref{n=2k+1} and the same index of whole group $SL\left( 2k+1,p \right)$ namely $\left[  ESL\left( 2k+1,p \right): \right. \\ \left.  SL\left( 2k+1,p \right)  \right]=2$. The isomorphism $\rho$ is justified in Remark \ref{SLn=2k+1}.

\begin{thm}
 For $n=2k$, $q=p^m$, provided $(-1)^{\left( \frac{q-1}{gcd(q-1,n)} \right)}=1$, we have $\left[ Z\left( ESL\left( n,q \right) \right)\,\,:\,\,Z\left( SL\left( n,q \right) \right) \right]=2$, and $PESL\left( n,q \right)=PSL\left( n,q \right)$. Furthermore $ESL\left( n,q \right)$ is quasimple group in this case.  
\end{thm}

\begin{proof} The center $Z\left( ESL\left( n,q \right) \right)$ consists of scalar matrices so we consider an equation $x^n=\pm 1$.

The proof is based of the Fermat's theorem and the fact that multiplicative group of $\mathbb{F}_{q}$ is cyclic so $g^{{p-1}}=1$, therefore  $g^{\frac{p-1}{2}}=-1$ where $g$ is generator. This entails $(g^{\left( \frac{q-1}{2gcd(q-1,n)} \right)})^n=-1$ which implies that $\frac{q-1}{gcd(q-1,n)}$ is even too $\frac{q-1}{2gcd(q-1,n)} \in \mathbb{N}$.

The necessary of this condition follows from that solution of equation $x^m=-1$ and from the equation obtained by exponenting of this equation to $\frac{q-1}{2gcd(q-1,m)}$ power $(x^m)^{\left( \frac{q-1}{2gcd(q-1,m)} \right)^n}=(x^m)^{t(p-1)}= 1$ by Fermat's theorem. Therefore fraction $\frac{q-1}{gcd(q-1,m)}$ have to be even for existence of solution in $\mathbb{F}_q$.
\end{proof}

\begin{center}
Diagram, for the case $ESL\left( n, q\right)$, $\mathbb{F}_q, q=p^m$,  $(-1)^{\left(\frac{q-1}{\gcd(n, q-1)} \right)}=1$, $n= 2k$
\end{center}
\begin{equation}\label{diag:my_diagram}
\xymatrix{
SL(n, \mathbb{F}_q) \ar @{^{(}->}@<-1ex>^{\phi} [r] \ar@{.>}[dr]|-{(i)}  \ar @{->>} [d] ^{\rho } 
& ESL(n,\mathbb{F}_q)  \ar @{->>}  [d] ^{\psi}  \\
PSL(n,\mathbb{F}_q)                       & PESL(n, \mathbb{F}_q) \ar @{<<->>}[l]_{=} }
\end{equation}

\begin{thm}
 For $n=2k$, $q=p^m$, provided $(-1)^{\left( \frac{q-1}{gcd(q-1,n)} \right)}=-1$, we have
 then  $Z(ESL(2,q)) \simeq Z(SL(2,q))$ for $k \in \mathbb{N}$, and furthermore, $PSL\left( 2, q \right)\triangleleft PESL\left( 2, q \right)$.
 
\end{thm}

\begin{center}
\textbf{For $\left( \frac{-1}{p} \right)=-1$ (The case $n= 2k$)}
\end{center}
\begin{equation}\label{diag:my_diagram}
\xymatrix{
SL(n, \mathbb{F}_p) \ar @{^{(}->}@<-1ex>^{\phi} [r] \ar@{.>}[dr]|-{(i)}  \ar @{->>} [d] ^{\rho }
& ESL(n,\mathbb{F}_p)  \ar @{->>}  [d] ^{\psi}  \\
PSL(n,\mathbb{F}_p)   \ar @{^{(}->}@<-1ex>^{\xi} [r]                 & PESL(n, \mathbb{F}_p)  }
\end{equation}
where epimorphism $\psi$ is homomorphism with the kernel $N = <E, -E>$.
An epimorphism $\rho$ has kernel subgroup of scalar matrices with $det(A)=1$.



\begin{theorem}\label{n=2k+1}
For $n=2k+1$ we have $PESL\left( n,p \right)=PSL\left( n,p \right)$ an index of a center $\left[ Z\left( ESL\left( 2k+1,p \right) \right) : \right. \\ \left. Z\left( SL\left( 2k+1,p \right) \right) \right]=2$. Furthermore, $ESL\left( n, p\right)$ is quasisimple, except for special cases of $( n, p ) \in \{(2,2); (2,3)\}$. 
\end{theorem}

\begin{center}
\textbf{Commuting diagram for $n=2k+1$ }
\end{center}
\begin{equation}\label{diag:3,Z}
\xymatrix{
SL(2k+1, \mathbb{Z}) \ar @{^{(}->}@<-1ex>^{\phi} [r] \ar@{.>}[dr]|-{(i)}   \ar @{<->}[d]_{\simeq}^{\rho}
& ESL(2k+1,\mathbb{Z}) \ar @{->>}@<-1ex>[l]_{\mathfrak{C}} \ar @{->>}  [d] ^{\psi}  \\
PSL(2k+1,\mathbb{Z})                        & PESL(2k+1, \mathbb{Z}) \ar @{<<->>}[l]_{=} }
\end{equation}

\begin{proof}
 An epimorphism $\rho$ has the kernel subgroup of scalar matrices with $det(A)=1$.
In the general case $n=2k+1$, there is an isomorphism $PESL(2k+1, \mathbb{Z}) \simeq PSL(2k+1, \mathbb{Z})$ due to the absence of scalar matrices $A_i \in SL(2k+1, \mathbb{Z})$ with $det(A_i)=-1$, belonging to $Z(SL(2k+1, \mathbb{Z}))$. This entails $| Z( ESL(2k+1,\mathbb{Z})):  Z(SL(2k+1,\mathbb{Z}))|=2$ which producing the mentioned above $PESL(2k+1, \mathbb{Z}) \simeq PSL(2k+1, \mathbb{Z})$.

In accordance with Definition the group \ref{quasimple} $ESL\left( n, p\right)$ is quasisimple, because of its quotient by a center is simple group in this case.  
\end{proof}
 
In the dimension 2, we have a commutative diagram of morphisms

\begin{center}
$\xymatrix{\ar @{} [dr] |{}
SL(2,\mathbb{Z}) \ar @{^{(}->}^{\phi} [r]  \ar @{->>} [d] ^{ } & ESL(2, \mathbb{Z}) \ar @{->>}  [d] ^{\psi} \ar@{.>}[dl]|-{(i)} \\
PSL(2,\mathbb{Z}) \ar @{^{(}->}^{\xi} [r]    & PESL(2,\mathbb{Z}) }$
\end{center}

where epimorphism ${\psi}$ provides us the quotient ${}^{ESL(2,\mathbb{Z})}/{}_{ <E, -E>} \simeq PSL(2,\mathbb{Z})$.


\begin{rem}
 There is not surjective homomorphism from $ESL_2(\mathbb{Z})$ to $SL_2(\mathbb{Z})$.
\end{rem}
\begin{proof}
Taking into account that $ESL_2(\mathbb{Z})$ can be generated by involutions see the section 4 then all its elements under homomorphism $\phi'$ from $ESL_2(\mathbb{Z})$ to $SL_2(\mathbb{Z})$ maps in elements of second order in $SL_2(\mathbb{Z})$, but there are only $E$ and $-E$ of order in $SL_2(\mathbb{Z})$. Thus, there is not surjective homomorphism from $ESL_2(\mathbb{Z})$ to $SL_2(\mathbb{Z})$.
\end{proof}

\begin{prop}
    If $p=4k+1$ then $| \mathbb{Z}(PESL(2,p)):\mathbb{Z}(PSL(2,p))|=2$. 
\end{prop}
\begin{proof}
   Given $p=4k+1$, then $(\frac{-1}{p})=1$ viz. $-1$ is a quadratic residue in the characteristic field $p$, therefore, the elements of subgroup of scalar matrices of $ESL(2k, p)$ takes the form $C= \left( \begin{matrix}
   a & 0  \\
   0 & a  \\
\end{matrix} \right), a \in \mathbb{F}_p$ with $detC= \pm1$. The elements of the centre of $ESL(2k, p)$ have the same form  with the only difference that their $detC= \pm1$. Taking into account the uniform distribution of residues and non-residues in $\mathbb{F}_p$ one immediately deduces that equation $a^2= \pm1$ possess 
solutions that are twice as multiple as $a^2= 1$.
         Consequently $|\mathbb{Z}(PESL(2k,p)):\mathbb{Z}(PSL(2k,p))|=2$, $k \in \mathbb{N}$ and $PESL(2k,p)=PSL(2k,p)$.
         
         Similar corollary is true that is if $p=4k+1$ then $[\mathbb{Z}(PESL(2k,p)):\mathbb{Z}(PSL(2k,p))]=2$, $k \in \mathbb{N}$.
\end{proof}

\begin{prop}
    If $p=4k+3$ then $|\mathbb{Z}(PESL(2k,p)):\mathbb{Z}(PSL(2k,p))|=1$, $k \in \mathbb{N}$, in particular $2|PSL(2,p)|=|ESL(2,p)|$.     
\end{prop}
 \begin{proof}
Since the centre $ESL(2,p)$  consists of scalar matrices, it can only contain matrices of the form 
$\left( \begin{matrix}
   a & 0  \\
   0 & a  \\
\end{matrix} \right)$ but if $-1$ is not a residue in ${{\text{F}}_{p}}$, then there are no solutions in $F_p$ for the equation ${{a}^{2}}\equiv -1(\bmod p)$, so there is only one matrix $E$ in the centre $Z(ESL_2(\mathbb{F}_p))$. The same is true in any even dimension $n=2k$.
\end{proof}

\section{Involutive generating set.}
 
Let ${{i}_{1}}=\left( \begin{matrix}
   -1 & 1  \\
   0 & 1  \\
\end{matrix} \right)$, ${{\rho}}=\left( \begin{matrix}
   0 & 1  \\
   1 & 0  \\
\end{matrix} \right)$ and ${{D}_{1}}=\left( \begin{matrix}
   -1 & 0  \\
   0 & 1  \\
\end{matrix} \right)$ now we show that $\{{{i}_{1}}, \rho , {{D}_{1}}\}$ forms involutive generating set of ${{ESL}_2(\mathbb{Z})}$. 
To obtain the transvection $t_{12}$ we change a sign in the first column of $i_1$ by multiplying ${{i}_{1}}{{D}_{1}}=\left( \begin{matrix}
   1 & 1  \\
   0 & 1  \\
\end{matrix} \right)$. 
Applying conjugation by $\rho$ we express 
$ \rho t_{12} \rho =  \left(
\begin{matrix}
   1 & 0  \\
   1 & 1  \\
\end{matrix} \right) = t_{21}$ that is second transvection from ${SL_2(\mathbb{Z})}$.  

As well known the transvections $t_{12}$ and $t_{21}$ generate group $S{{L}_{2}}\left( \mathbb{Z} \right)$ therefore an arbitrary element $C \in SL_2(\mathbb{Z})$ can be expressed. Consider the matrix equation ${{D}_{1}}X =C$ we see that $X\in ESL_2(\mathbb{Z})$. 
Due to property of uniqueness decision of group equation the bijective correspondence between different matrices $X_i, X_j \in SL_2(\mathbb{Z})$ and solutions of equation $X_1, X_2$ holds. Thus all matrices with determinant equals to -1 are present in the group closure of $\{{{i}_{1}}, \rho , {{D}_{1}}\}$.

Since transvections ${{t}_{21}}=\left( \begin{matrix}
   1 & 1  \\
   0 & 1  \\
\end{matrix} \right)$ and ${{t}_{12}}=\left( \begin{matrix}
   1 & 0  \\
   1 & 1  \\
\end{matrix} \right)$ is generating set for $S{{L}_{2}}\left( \mathbb{Z} \right)$ it remains to form splittable extension of $S{{L}_{2}}\left( \mathbb{Z} \right)$ by top subgroup $D_1$.

\textbf{Corollary}. The involutive generating set with 3 non-commutative involutions for $ES{{L}_{2}}\left( \mathbb{Z} \right)$ is $S=\left\langle {{i}_{1}},\,\,{{\rho}_{}},\,\,{{D}_{1}} \right\rangle $.

\begin{prop}
   The group $ES{{L}_{2}}\left( \mathbb{Z} \right)$ does not possess a generating involutive set of type $(2 \times 2, 2)$.
\end{prop}
The unique involution commuting with rest of involutions of $ES{{L}_{2}}\left( \mathbb{Z} \right)$ is $-E$. So without $-E$ the a triple of generating involutions of the type $(2\times2, 2)$ can not be formed.
Let us analyse acting of diagonal involutions on upper-triangle involution of form ${{i}_{1}}=\left( \begin{matrix}
   -1 & 1  \\
   0 & 1  \\
\end{matrix} \right)$. The similar acting of diagonal involutions will be on ${{i}_{4}}=\left( \begin{matrix}
   1 & 1  \\
   0 & -1  \\
\end{matrix} \right)$ and on lower-triangle ${{l}_{1}}=\left( \begin{matrix}
   -1 & 0  \\
   1 & 1  \\
\end{matrix} \right)$, ${{l}_{4}}=\left( \begin{matrix}
   1 & 0  \\
   1 &-1  \\
\end{matrix} \right)$ and on the same involutions in the form of matrices with a nonzero additional diagonal, which can be expressed as $\rho {{i}_{1}}$ from the involution listed above. Therefore, it is sufficient to study the group closure of involutions ${{i}_{1}}$ and $D_1$ provided $-E$ is present in generic set.

Let us show that it is impossible to express both transvections $t_{12}$, $t_{21}$ by the involutions $i_1$ and $D_1$ as well as it is impossible to generate $-t_{12}$, $-t_{21}$ which are related to $t_{12}=-E(-t_{12})$, $t_{21}=-E(-t_{21})$.

Consider action of $D_1$ from the right changing sign of first row ${{i}_{1}}{{D}_{1}}=\left( \begin{matrix}
   1 & 1  \\
   0 & 1  \\
\end{matrix} \right)$. Acting from the left changes sign of first row ${{D}_{1}} {{i}_{1}} =\left( \begin{matrix}
   1 & -1  \\
   0 & 1  \\
\end{matrix} \right) = t_{12}^{-1}$. But $t_{12}^{}$ is already expressed thus by group axioms $t_{12}^{-1}$ is already expressed too. 
It remains to consider ${{D}_{1}} {{i}_{1}} {{D}_{1}}$  manipulations with it. The result of involution conjugation is again involution ${{D}_{1}} {{i}_{1}} {{D}_{1}} = \left( \begin{matrix}
   -1 & -1  \\
   0 & 1  \\
\end{matrix} \right)$. Successive multiplication of this involution by $D_1$, both from the left and from the right by $i_{1}$, returns us to the matrices obtained above, for instance ${{D}_{1}} {{D}_{1}} {{i}_{1}} {{D}_{1}}= {{i}_{1}} {{D}_{1}} =t_{12}$ because of $D_1$ is involution. 

Alternative set of 3 involutions $\left\langle {{i}_{1}},\,\,{{\rho}_{}},\,\, -E \right\rangle $ does not lead us to expressing both of transvections, in view of one minus in ${i}_{1}$ can not be reduced by $-E$ as well as action of rotation by $\rho$. 
Consequently, a matrix with three elements 1 and the fourth 0 cannot be reached.

At least the conjugation of $t_{12}$ equals $t_{12}^{D_1}$, does not lead us to a new transvection, since ${{D}_{1}} ({{t}_{12}} {{D}_{1}}) = {{D}_{1}} ({{i}_{1}})= {{D}_{1}} {{i}_{1}} = \left( \begin{matrix}
1 & -1 \\
0 & 1 \\
\end{matrix} \right) = t_{12}^{-1}$, which is already obtained above.

By the same reasoning any two transvections of generic set ${{R}_{1}}=\left( \begin{matrix}
   0 & 1  \\
   1 & 0  \\
\end{matrix} \right)$, ${{R}_{2}}=\left( \begin{matrix}
   1 & 0  \\
   1 &-1  \\
\end{matrix} \right)$,
${{R}_{3}}=\left( \begin{matrix}
   -1 & 0  \\
   0 & 1  \\
\end{matrix} \right)$ do not generate whole $ESL(2,Z)$ even these two generators will be completed by $-E$.



{\bf Proposition 2.} {\ 
The \textbf{minimal generating set} of $ESL_3[\mathbb{Z}]$ is $\langle {M}_{6}, t_{12}\rangle$
\begin{center}
  ${{M}_{6}}=\left( \begin{matrix}
   0 & 1 & 0  \\
   0 & 0 & 1  \\
   -1 & 0 & 0  \\
\end{matrix} \right)$ and $t_{12}=\left( \begin{matrix}
   1 & 1 & 0  \\
   0 & 1 & 0  \\
   0 & 0 & 1  \\
\end{matrix} \right),$ 
\end{center}
which possess the relations ${M}_{6}t_{12} {M}^{-1}_{6}=t^{-1}_{31}, \, {M}_{6}t_{31} {M}^{-1}_{6}=t^{-1}_{23} $, ${M}_{6}t_{23} {M}^{-1}_{6}=t^{-1}_{12}$, ${M}_{6}t_{13}{M}^{-1}_{6}=t^{-1}_{32}$, ${M}_{6}t^{-1}_{23} {M}^{-1}_{6}=t^{-1}_{31}$, ${M}^6_{6}=E$, ${M}^3_{6}=-E$, as well as the relations between transvections relations between transvections $[t_{ij}, t_{jk}]=t_{ik}, wherein \, i\neq k, and \,\, [t_{ij}, t_{kl}]=e$ for $i\neq l$ and $k \neq j$ .}

Proof. Due to the relations mentioned above all elementary transvections are presented in explicit form, thence in accordance with \cite{Humph} and Lemma \ref{negativedet} (about an additional matrix with negative determinant) the generating set of $ESL_3[\mathbb{Z}]$ is constructed.

The order of the permutation matrix is indicated by the relation $P^5=E$.
The size of generic set is minimal for non-cyclic group so its minimality does not order a proof.
The size of generated  group is 11232 is in 2 times grater then order of $SL(3,\mathbb{Z})$.

\section{Minimal involutive generating set of type $(2\times2, 2)$ for $ESL_2(Z)$ and involutive sets of generators for $ESL_3(Z)$. }

In 1980, question 7.30 was added to the Kourovka Notebook \cite{Maz}: Which finite simple groups are generated by three involutions, two of which are permutable? Groups with this property were later called (2 × 2, 2)-groups.

Consider the minimal generating set for $ESL_2[{\mathbb{Z}}]$: $S=  \langle t_{12}, \rho  \rangle$, where ${{\rho}}=\left( \begin{matrix}
   0 & 1  \\
   1 & 0  \\
\end{matrix} \right)$ and ${{t}_{12}}=
\left( \begin{matrix} 1 & 1  \\
                      0 & 1  \\
\end{matrix} \right)$ that is transvection. Conjugation of $t_{11}$ lead as to $\rho  t_{12} \rho =t_{21}$. The rest of transvection can be reached by rotation $\rho t_{12}= \left( \begin{matrix} 0 & 1  \\
                      1 & 1  \\
\end{matrix} \right)=t_{22}$.
A very impotent property of involutions is formulated below.
\begin{lem}
    If involutions $A$, $B$ commute then $AB$ is involution too. Otherwise is also true.

Proof.
Let $i_1i_2=i_2i_1$ then $i_1i_2i^{-1}_1=i_2$, $i_1i_2i^{-1}_1i^{-1}_2=e$ and so $i_1i^{-1}_1i_2i_2=i_1i_1i_2i_2=(i_1i_2)^2=e.$

Vice versa if $(i_1i_2)^2=e$ then $(i_1i_2) (i_1i_2)=i_1i_2 i^{-1}_1i^{-1}_2=e$ so $[i_1,i_2 ]=e$.
\begin{exm}
    Let
${{i}_{12}}=\left( \begin{matrix}
   -1 & 1 & 0  \\
   0 & 1 & 0  \\
   0 & 0 & 1  \\
\end{matrix} \right)$
and
${{i'}_{23}}=\left( \begin{matrix}
   -1 & 1 & 0  \\
   0 & 1 & 1  \\
   0 & 0 &-1  \\
\end{matrix} \right)$, 
 then ${{i}_{12}}{i'}_{23}= {{i'}_{23}}{i}_{12} = \left( \begin{matrix}
   1 & 0 & 1  \\
   0 & 1 & 1  \\
   0 & 0 &-1  \\
\end{matrix} \right) = I$ that is involution $I^2=E$.
\end{exm}
\end{lem}

\begin{theorem}\label{minimal3gen} 
A minimal generating set for $ESL_n(\mathbb{F}_2)$, $ESL_n(\mathbb{F}_p)$ with $n > 2$ and $p \in \mathbb{P}$ contains at least 3 generators.

The exceptional case $ESL_2(\mathbb{F}_2)$ is 2-generated group, moreover \\ $ESL_2(\mathbb{F}_2) \simeq D_3 \simeq \langle \rho, t_{12} \rangle$. 
\end{theorem}
\begin{proof}
 Due to the well known isomorphism $SL_2(\mathbb{F}_2) \simeq S_3 \simeq D_3 $ and the fact that $1=-1$ in $\mathbb{F}_2$ we have $SL_2(\mathbb{F}_2)= ESL_2(\mathbb{F}_2)\simeq D_3$. As a direct consequence, two involutions generate $ESL_2(\mathbb{F}_2)$. For instance, $ESL_2(\mathbb{F}_2) \simeq \langle \rho, t_{12} \rangle$, note that $t_{12}$ is the involution in $ESL_2(\mathbb{F}_2)$.

To show that $SL_2(\mathbb{F}_p)$ is not a group generated by two involutions, like the dihedral group we show an absence of isomorphism $ESL_2(\mathbb{F}_p)$ with $D_{2p}$.

In order to show that $ESL_2(\mathbb{F}_p)$ is not two involutions generated group as dihedral group, we show an absence of isomorphism $ESL_2(\mathbb{F}_p)$ with $D_{2p}$. 

The order of $ESL_2(\mathbb{F}_3)$ is 48.
 $ESL_2(\mathbb{F}_3)$ does not contain an element of order 24, hence it cannot be isomorphic to $D_{48}$ having a cyclic group of order 24. Similarly $SL_2(\mathbb{F}_3)$ has no elements with order 12 so $SL_2(\mathbb{F}_3)$ is not isomorphic to $D_{12}$. Arguing in similar way we justify that $ESL_2(\mathbb{F}_p)$ has not two generating set. 

 If $n>3$ then $D_{2p}$ is solvable in contrast with $ESL_2(\mathbb{F}_p)$. That completes the proof.
  \end{proof}

For convenience we fix some notations for \textit{diagonal involutive matrices} from $ESL_3(\mathbb{Z})$:
\begin{center}  
$I_{12} = \begin{pmatrix}
            -1 & 1 & 0 & 0 & 0 \\
            0  & 1 & 0 & 0 & 0 \\
            0  & 0 & 1 & 0 & 0 \\
            0  & 0 & 0 & 1 & 0 \\
            0  & 0 & 0 & 0 & 1
        \end{pmatrix},
        I_{23} = \begin{pmatrix}
            1 & 0 & 0 & 0 & 0 \\
            0 & -1& 1 & 0 & 0 \\
            0 & 0 & 1 & 0 & 0 \\
            0 & 0 & 0 & 1 & 0 \\
            0 & 0 & 0 & 0 & 1
        \end{pmatrix}, \ldots,
        I_{45} = \begin{pmatrix}
            1 & 0 & 0 & 0 & 0 \\
            0 & 1 & 0 & 0 & 0 \\
            0 & 0 & 1 & 0 & 0 \\
            0 & 0 & 0 & -1& 1 \\
            0 & 0 & 0 & 0 & 1
        \end{pmatrix}.
    $
    \end{center}

\textbf{Proposition 2.} {\sl The minimal involutive generating set for both $ESL_5\left[ \mathbb{Z}\right]$ and $ESL_5\left(  \mathbb{F}_p \right)$ consists of 3 involutions. One of such set is the following:}

\begin{center}
    $
        I_{12} = \begin{pmatrix}
            -1 & 1 & 0 & 0 & 0  \\
            0  & 1 & 0 & 0 & 0  \\
            0  & 0 & 1 & 0 & 0  \\
            0  & 0 & 0 & 1 & 0  \\
            0  & 0 & 0 & 0 & 1
        \end{pmatrix},
        D_1 = \begin{pmatrix}
            0  & 0  & 0  & 0  & -1 \\
            0  & 0  & 0  & 1 & 0  \\
            0  & 0  & -1 & 0  & 0  \\
            0  & 1 & 0  & 0  & 0  \\
            -1 & 0  & 0  & 0  & 0
        \end{pmatrix},
        F_L = \begin{pmatrix}
            1 & 0 & 0 & 0 & 0 \\
            0 & 0 & 0 & 0 & 1 \\
            0 & 0 & 0 & 1 & 0 \\
            0 & 0 & 1 & 0 & 0 \\
            0 & 1 & 0 & 0 & 1
        \end{pmatrix}.
    $
\end{center} The key step in the proof is to generate all such involutions using the given involutions, for instance
$F_L I_{12} F_L = I_{15}$. 

Proof. The key step in the proof is to generate all elementary transvections using the given involutions. 
In this set, there exist words for generating both a transvection with length 26 letters) and a permutation matrix $P$ of order 5 (generated by word of 22 letters in the alphabet $I_{12}, D_1, F_L$):
\begin{center}
     \begin{align*}
    &   t_{41} = I_{12} D_1 F_L I_{12} D_1 F_L I_{12} D_1 I_{12} D_1 F_L I_{12} D_1 I_{12} D_1 F_L I_{12} D_1 I_{12} D_1 F_L I_{12} D_1 F_L D_1 I_{12}, \\
    &    P = \begin{pmatrix}
            0 & 0 & 0 & 1 & 0 \\
            0 & 0 & 1 & 0 & 0 \\
            1 & 0 & 0 & 0 & 0 \\
            0 & 0 & 0 & 0 & 1 \\
            0 & 1 & 0 & 0 & 0
        \end{pmatrix} = I_{12} F_L D_1 F_L I_{12} D_1 I_{12} D_1 F_L D_1 F_L I_{12} D_1 F_L D_1 I_{12} F_L D_1 F_L I_{12} F_L D_1.
    \end{align*}
\end{center}
According to \cite{VSEM}, one transvection and a permutation matrix are sufficient to generate $SL(5, \mathbb{Z})$, by virtue of the action of a permutation matrix generates other transvections from a given one. Furthermore, after generating $SL(5, \mathbb{Z})$ one matrix $I_{12}$ ($|I_{12}| = -1$) is sufficient to extend it to the entire $ESL(5, \mathbb{Z})$, because of uniqueness solution of group equation $I_{12}X=Y$, $X \in SL_5[\mathbb{Z}]$, $Y \in ESL_5(\mathbb{Z})$.

\textbf{Theorem 1.} \label{ID0FU} {\sl The minimal involutive generating set of $ESL\left( 5, \mathbb{Z} \right)$ as well as for $ESL\left( 5, \mathbb{F}_p \right)$ consists of 3 involutions $D_0, F_U, I_{12}$ with the relations $(D_0 F_U)^{4} I_{12} (F_U D_0)^{4}=I_{12}$, $(I_{12} D_0)^4 =E$, $(I_{12} F_U)^4 =E$, $(D_0 F_U )^5=E$, $D_0^2 = F_U^2 = I^2_{12}=E$, where }
    \begin{center}
    $ 
        I_{12} = \begin{pmatrix}
            -1 & 1 & 0 & 0 & 0  \\
            0  & 1 & 0 & 0 & 0  \\
            0  & 0 & 1 & 0 & 0  \\
            0  & 0 & 0 & 1 & 0  \\
            0  & 0 & 0 & 0 & 1
        \end{pmatrix},
                D_0 = \begin{pmatrix}
            0  & 0  & 0  & 0  & -1 \\
            0  & 0  & 0  & -1 & 0  \\
            0  & 0  & -1 & 0  & 0  \\
            0  & -1 & 0  & 0  & 0  \\
            -1 & 0  & 0  & 0  & 0
        \end{pmatrix},
        F_U = \begin{pmatrix}
            0 & 0 & 0 & 1 & 0 \\
            0 & 0 & 1 & 0 & 0 \\
            0 & 1 & 0 & 0 & 0 \\
            1 & 0 & 0 & 0 & 0 \\
            0 & 0 & 0 & 0 & 1
        \end{pmatrix}.
              $
             \end{center}
            This involutions possess a relation of diagonal shift of involutive cell $D_0 F_U I_{12} F_U D_0=I_{23}$, which imply the relation $(D_0 F_U)^{4} I_{12} (F_U D_0)^{4}=I_{12}$, $I_{12} D_0 I_{12} D_0 I_{12} D_0 I_{12} D_0 =E$, $(I_{12} F_U)^4 =E$, $(D_0 F_U )^5=E$.
The key step in the proof is to generate a transvection using the given involutions, in order to do this we investigate the relation in this generic set. Then $t_{25}$ is expressed by word of 26 elements:
\begin{equation*}
t_{25} = D_0 I_{12} F_U I_{12} D_0 I_{12} F_U I_{12} D_0 I_{12} D_0 F_U I_{12} D_0 I_{12} D_0 F_U I_{12} D_0 I_{12} D_0 F_U D_0 F_U I_{12} D_0.
\end{equation*}
Constructing this set consisting of $n$ transvections according to Theorem 2.1 from \cite{Humph} means that we have constructed a generating set. Furthermore as it was studied in \cite{VSEM} a minimal generating set from transvections of $SL(n,\mathbb{Z})$ has size $n$. This set of generators allows us to express the permutation matrix $P_5$ in the form
\begin{equation*}
    P_5 = F_U \cdot D_1 \cdot F_U \cdot D_1 =
    \begin{pmatrix}
        0 & 0 & 1 & 0 & 0 \\
        0 & 0 & 0 & 1 & 0 \\
        0 & 0 & 0 & 0 & 1 \\
        1 & 0 & 0 & 0 & 0 \\
        0 & 1 & 0 & 0 & 0
    \end{pmatrix}.
\end{equation*}


\begin{lem}\label{negativedet}
  \textit{Let $A_1, A_2, ..., A_k \in ESL(n, \mathbb Z)$ be a set of matrices, where at least one matrix $A_i$ has a negative determinant. If $SL(n, \mathbb Z) \subseteq G = \left\langle A_1, A_2, ..., A_k \right\rangle$, then $G = ESL(n, \mathbb Z)$.}    
\end{lem}
\textbf{Proof.} Take an arbitrary matrix $B \in SL^{-}(n, \mathbb Z)$, and write it in the form $B = A_i^{-1} A_i B$. The matrix $A_i B$ belongs to $SL(n, \mathbb Z)$, and therefore a word for the matrix $B$ exists in $G$.

\textbf{Theorem 2.} \label{ID1FL} {\sl The minimal involutive generating set of $ESL\left( 5, \mathbb{Z} \right)$ as well as for $ESL\left( 5, \mathbb{F}_p \right)$ could be constructed of the following 3 involutions
\begin{equation*}
    I_{12} = \begin{pmatrix}
        -1 & 1 & 0 & 0 & 0  \\
        0  & 1 & 0 & 0 & 0  \\
        0  & 0 & 1 & 0 & 0  \\
        0  & 0 & 0 & 1 & 0  \\
        0  & 0 & 0 & 0 & 1
    \end{pmatrix},
    D_1 = \begin{pmatrix}
        0  & 0  & 0  & 0  & -1 \\
        0  & 0  & 0  & 1 & 0  \\
        0  & 0  & -1 & 0  & 0  \\
        0  & 1 & 0  & 0  & 0  \\
        -1 & 0  & 0  & 0  & 0
    \end{pmatrix},
    F_L = \begin{pmatrix}
        1 & 0 & 0 & 0 & 0 \\
        0 & 0 & 0 & 0 & 1 \\
        0 & 0 & 0 & 1 & 0 \\
        0 & 0 & 1 & 0 & 0 \\
        0 & 1 & 0 & 0 & 0
    \end{pmatrix}.
\end{equation*}
}
Proof.  In this set, there exist words for generating both a transvection (26 letters) and a permutation matrix of order 5 (22 letters):
\begin{center}
    \begin{align*}
        t_{41} &= I_{12} D_1 F_L I_{12} D_1 F_L I_{12} D_1 I_{12} D_1 F_L I_{12} D_1 I_{12} D_1 F_L I_{12} D_1 I_{12} D_1 F_L I_{12} D_1 F_L D_1 I_{12}, \\
        P &= \begin{pmatrix}
            0 & 0 & 0 & 1 & 0 \\
            0 & 0 & 1 & 0 & 0 \\
            1 & 0 & 0 & 0 & 0 \\
            0 & 0 & 0 & 0 & 1 \\
            0 & 1 & 0 & 0 & 0
        \end{pmatrix} = I_{12} F_L D_1 F_L I_{12} D_1 I_{12} D_1 F_L D_1 F_L I_{12} D_1 F_L D_1 I_{12} F_L D_1 F_L I_{12} F_L D_1.
    \end{align*}
\end{center}
According to \cite{VSEM} one transvection and a permutation matrix  are sufficient to generate $SL(5, \mathbb{Z})$, and since $|D_1| = -1$, the entire $ESL(5, \mathbb{Z})$ is generated.
And since the system of generators contains $D_1$ having $|D_1| = -1$, then the entire $ESL(5, \mathbb{Z})$ is generated by Lemma \ref{negativedet}.

\textbf{Corollary 2.} {\sl
If we replace $D_0$ with $-D_0$, that is, 
we introduce another matrix with a negative determinant,
then the resulting set $\{I_{12}, -D_0, F\}$ also generates the whole $ESL(5, \mathbb Z)$.
}
 
Proof.
Consider this set of matrices $F, -D_0, I_{12}$:

\begin{equation*}
    I_{12} = \begin{pmatrix}
        -1 & 1 & 0 & 0 & 0  \\
        0  & 1 & 0 & 0 & 0  \\
        0  & 0 & 1 & 0 & 0  \\
        0  & 0 & 0 & 1 & 0  \\
        0  & 0 & 0 & 0 & 1
    \end{pmatrix},
    -D_0 = \begin{pmatrix}
        0  & 0  & 0  & 0  & 1 \\
        0  & 0  & 0  & 1 & 0  \\
        0  & 0  & 1 & 0  & 0  \\
        0  & 1 & 0  & 0  & 0  \\
        1 & 0  & 0  & 0  & 0
    \end{pmatrix},
    F_U = \begin{pmatrix}
        0 & 0 & 0 & 1 & 0 \\
        0 & 0 & 1 & 0 & 0 \\
        0 & 1 & 0 & 0 & 0 \\
        1 & 0 & 0 & 0 & 0 \\
        0 & 0 & 0 & 0 & 1
    \end{pmatrix}.
\end{equation*}

Firstly, we write the expression for permutation matrix $P$
\begin{equation*}
    P = F_U \cdot D_0 \cdot F_U \cdot D_0 =
    \begin{pmatrix}
        0 & 0 & 1 & 0 & 0 \\
        0 & 0 & 0 & 1 & 0 \\
        0 & 0 & 0 & 0 & 1 \\
        1 & 0 & 0 & 0 & 0 \\
        0 & 1 & 0 & 0 & 0
    \end{pmatrix}.
\end{equation*}

It turns out that the set of matrices $F_U, -D_0, I_{12}$ can also generate the transvection $T_{41}$, namely by the following word of 24 elements:
\begin{equation*}
    T_{41} = I_{12}F_U I_{12}(-D_0) I_{12}F_U I_{12}(-D_0) I_{12}(-D_0) F_U I_{12}(-D_0) I_{12}(-D_0) F_U I_{12}(-D_0) I_{12}(-D_0) F_U (-D_0) F_U I_{12}.
\end{equation*}
So we may assert that in this no transvection generating set can be obtained by a word of length less than 24. Thus the matrices $F_U, -D_0, I_{12}$ generate $ESL(5, \mathbb Z)$.

\textbf{Remark.} \label{FL}
The transformation to previous generating set can be given by the formula: \begin{equation*}
    F_L = D_0 F_U D_0
\end{equation*}
which proving that set $F_L, D_0, I_{12}$ is generating set too as well as $F_U, D_0, I_{12}$, because we made only equvivalent invertible 
Nielsen transformations of generators.
Thus, instead $F_U$ we can use the matrix $F_L$.

\textbf{Corollary 1.} {\sl
The minimal involutive generating sets for both $ESL_5[\mathbb{Z}]$ and $ESL_5\left( \mathbb{F}_p \right)$ consist of 3 involutions, provided that $n \equiv 1(mod{4})$.}

\textbf{Theorem.} {\sl
A minimal involutive generating sets of $ESL_n(\mathbb{Z}), \, n=5, k\in \mathbb{N}$ as well of $ESL_5\left(  \mathbb{F}_p \right)$ consist of 3 following involutions $ \l D_1, F_U, I_{12} \r$.}

Proof.  As was studied in \cite{VSEM} a minimal generating from transvections set of $ESL(n,\mathbb{Z})$ has size $n$. 
The construction of the transition to the two-element set $\left\langle P, t_{13} \right\rangle $ similar to the minimal generating set from Proposition 1 of generators $P$ and elementary transvection $t_{13}$ is written in the following expressions $t_{13} = I_{12} D_1 F_U I_{12} F_U D_1 I_{12} D_1 F_U I_{12} F_U D_1$, 
\begin{center}
         $P_5 = D_1 F D_1  F = \begin{pmatrix}
        0 & 0 & 0 & 1 & 0 \\
        0 & 0 & 0 & 0 & 1 \\
        1 & 0 & 0 & 0 & 0 \\
        0 & 1 & 0 & 0 & 0 \\
        0 & 0 & 1 & 0 & 0
    \end{pmatrix}$.
\end{center}
This completes the proof.

\textbf{Proposition.} \label{Growth} {\sl
There are no generic set of three involutions for $ESL(n,\mathbb{Z})$, one of which commutes with two others.}

\textbf{Proof}.
A number of matrices with the norm less or equal than $R$ be denoted by $N(R)$.
All words over such generating set $\langle A, B, C \rangle$ that can be constructed over an alphabet from two involutions $A, B$ ($[A, B]$) are represented by one of the next sequences either ${{\text{(}AB\text{)}}^{n}}$ or $B{{(AB)}^{n}}$ or ${{(AB)}^{n}}A$ or $B{{(AB)}^{n}}A$.
It is sufficient to describe the form of the periodic part of these words, viz., ${{(AB)}^{n}}$ then it will be clear that not all the matrices from $ES{{L}_{2}}\left( \mathbb{Z} \right)$ can be represented in this way.

For the proof, a specific form ${{(AB)}^{n}}$ is needed.
An additional part of the proof is based on the fact that the group $ES{{L}_{2}}\left( \mathbb{Z} \right)$ is a two-parameter family. Therefore, it is unlikely to be covered by one-parameter families such as the sequence $AB$.

In view of $AB$ has a determinant equal to $±1$, then it is either similar to $\left( \begin{matrix}
   \pm 1 & 1  \\
   0 & \pm 1  \\
\end{matrix} \right)$  or diagonalizable, and in the last case at least one of the eigenvalues has modulus  $\ge 1$,
however over the ring $\mathbb{Z}$ in a diagonal matrix there only can be elements 1 or -1 in any arrangement.
(If both eigenvalues of $AB$  have modulus 1, and it is diagonalizable but then such a matrix $AB$ will simply have finite order. Such matrix generate cyclic group.)

For a matrix with a Jordan block size of 2 by 2 e.v. are $1$ or $-1$, thence $N(R)=cR$. Since the $n$-th power of such Jordan block is equal to ${AB}^n=\left( \begin{matrix}
1 & n \\
0 & 1 \\
\end{matrix} \right)$ or $\left( \begin{matrix}
-1 & n \\
0 & -1 \\
\end{matrix} \right)$.

Since the $n$-th power of the Jordan block is equal to $(AB)^n$, exactly the first $R$ terms (in a series of power of $(AB)$) are contained in a ball of radius $R$.
This cause that exactly the first $R$ terms are contained in a ball of radius $R$.
This determines the linear growth of the matrix norm $clog _{\alpha} R$.


At the same time, the number of elements $M$ where $M \in ESL_2(\mathbb{Z})$ grows as a function $R \sqrt{R}$, because if we fix first column of $M$, it remains to choose two elements of the second column by $\sqrt{R}$ ways in accordance with the theorem about the prime numbers distribution in order to  Diophantine equation $x \alpha - y \beta =1$ to be solvable over $\mathbb{Z}$.
Note that solvability of $x \alpha - y \beta =1$ is equivalent to $( \alpha, \beta)=1$ and therefore the theorem about prime numbers distribution is applicable to estimating the solutions number.


If the matrix $A \in ESL_2(\mathbb{Z})$ possesses the diagonal form $A=\left( \begin{matrix}
\alpha & 0 \\
0 & \alpha^{-1} \\
\end{matrix} \right)$, where $ \alpha \in \mathbb{C}$ and $| \alpha | >1$, $|\alpha^{-1} | <1$,
this yields exponential growth of a matrix $(AB)^n$ norm,
then in the series in its powers $A^n$ the following growth function of the number of matrices with norm no greater than $R$ holds: $N(R)=clog_{\alpha}R$, where constant $C$ appears due to equivalence transformation.

Therefore, ${{\left( AB \right)}^{n}}$ yields either linear or exponential growth.

The number of $ESL_2(\mathbb{Z})$ elements with a norm that is no grater than $R$ can be estimated from below as $R^2$ due to the relation on the determinant of these elements decrease $N(R)$ from $R^{k^2}$ to $R^{k^2-k}$.
Consequently, the number $N(R)$ of elements having norm no grater than $R$ generated by $\langle A, B, C \rangle$ which forms only four sequences (${{\text{(}AB\text{)}}^{n}}$, $B{{(AB)}^{n}}$, ${{(AB)}^{n}}A$, $B{{(AB)}^{n}}A$) is less than number elements of $ESL_2(\mathbb{Z})$ with a norm less than or equal to $R$, that's accomplish the proof.

{Corollary.}
{\  As a direct corollary we obtain that there is no two involution generating set of $ESL(n,\mathbb{Z})$, $n>2$.}

\section*{Formula for coordinate-wise action by conjugation}

For brevity, let us redesignate $F_U$ as $F$, $D_0$ as $D$.
Let's denote the operation $r_p(x) = -(x \mod {p}) + p$. That is, $r_p(x)$ is the same as taking the excess when dividing by $p$, only the excess is taken in the range from $1$ to $p$ instead of the usual $0$ to $p-1$.
Then the conjugation of the transvection $T_{ij}$ by the matrices $F, D$, and their composition can be written as formulas:

\begin{center}
    $F(i, j) = (r_5(-i), r_5(-j))$ \\
    $D(i, j) = (6 - i, 6 - j)$ \\
    $F \circ D(i, j) = ((r_5(i - 1)), r_5(j - 1))$ \\
    $D \circ F(i, j) = (6 - r_5(-i), 6 - r_5(-j))$
\end{center}

\section*{Commuting involutions of type $2\times2, 2$}
\begin{lemma}\label{minimal3gen} The exceptional case $ESL_2(\mathbb{F}_2)$ is 2-generated group, moreover $ESL_2(\mathbb{F}_2) \simeq D_3 \simeq \langle \rho, t_{12} \rangle$. 

A minimal generating set for $ESL_n(\mathbb{F}_2)$, $ESL_n(\mathbb{F}_p)$ with $n > 2$ contains at least 3 generators.
\end{lemma}
\begin{proof}
 Due to the well known isomorphism $SL_2(\mathbb{F}_2) \simeq S_3 \simeq D_3 $ and the fact that $1=-1$ in $\mathbb{F}_2$ we have $SL_2(\mathbb{F}_2)= ESL_2(\mathbb{F}_2)\simeq D_3$. As a direct consequence, two involutions generate $ESL_2(\mathbb{F}_2)$. For instance, $ESL_2(\mathbb{F}_2) \simeq \langle \rho, t_{12} \rangle$, note that $t_{12}$ is the involution in $ESL_2(\mathbb{F}_2)$.

To show that $SL_2(\mathbb{F}_p)$ is not a group generated by two involutions, like the dihedral group we show an absence of isomorphism $ESL_2(\mathbb{F}_p)$ with $D_{2p}$.

In order to show that $ESL_2(\mathbb{F}_p)$ is not two involutions generated group as dihedral group, we show an absence of isomorphism $ESL_2(\mathbb{F}_p)$ with $D_{2p}$. 

The order of $ESL_2(\mathbb{F}_3)$ is 48.
 $ESL_2(\mathbb{F}_3)$ does not contain an element of order 24, hence it cannot be isomorphic to $D_{48}$ having a cyclic group of order 24. Similarly $SL_2(\mathbb{F}_3)$ has no elements with order 12 so $SL_2(\mathbb{F}_3)$ is not isomorphic to $D_{12}$. Arguing in similar way we justify that $ESL_2(\mathbb{F}_p)$ has not two generating set. 

 If $n>3$ then $D_{2p}$ is solvable in contrast with $ESL_2(\mathbb{F}_p)$. That completes the proof.
  \end{proof}

\begin{theorem} \label{ID0FU} {\sl The minimal involutive generating set of $ESL\left( 5, \mathbb{Z} \right)$ as well as for $ESL\left( 5, \mathbb{F}_p \right)$ consists of 3 involutions $D_0, F_U, I_{12}$ with the relations $(D_0 F_U)^{4} I_{12} (F_U D_0)^{4}I_{12}^{-1}=E$, $(I_{12} D_0)^4 =E$, $(I_{12} F_U)^4 =E$, $(D_0 F_U )^5=E$, $D_0^2 = F_U^2 = I^2_{12}=E$, where }
    \begin{center}
    $
        I_{12} = \begin{pmatrix}
            -1 & 1 & 0 & 0 & 0  \\
            0  & 1 & 0 & 0 & 0  \\
            0  & 0 & 1 & 0 & 0  \\
            0  & 0 & 0 & 1 & 0  \\
            0  & 0 & 0 & 0 & 1
        \end{pmatrix},
                D_0 = \begin{pmatrix}
            0  & 0  & 0  & 0  & -1 \\
            0  & 0  & 0  & -1 & 0  \\
            0  & 0  & -1 & 0  & 0  \\
            0  & -1 & 0  & 0  & 0  \\
            -1 & 0  & 0  & 0  & 0
        \end{pmatrix},
        F_U = \begin{pmatrix}
            0 & 0 & 0 & 1 & 0 \\
            0 & 0 & 1 & 0 & 0 \\
            0 & 1 & 0 & 0 & 0 \\
            1 & 0 & 0 & 0 & 0 \\
            0 & 0 & 0 & 0 & 1
        \end{pmatrix}.
              $
             \end{center}
    
\end{theorem}
\begin{proof}   To justify the above relations, we note that these involutions possess the relation of diagonal shift of the involutive cell $D_0 F_U I_{12} F_U D_0=I_{23}$, which implies the relation $(D_0 F_U)^{4} I_{12} (F_U D_0)^{4}=I_{12}$, $I_{12} D_0 I_{12} D_0 I_{12} D_0 I_{12} D_0 =E$, $(I_{12} F_U)^4 =E$, $(D_0 F_U )^5=E$, also we take into account that conjugation by $D_0$ of $I_{i,i+1}$ change coordinates of unity that is ${(i,i+1)}$ on new coordinates $(d+1 -i, d+1 -i-1)$, where $d$ is dimension of matrix (in this case $d=5$).
The key step in the proof is to generate a transvection using the given involutions, in order to do this we investigate the relation in this generic set. Then $t_{25}$ is expressed by the word of 26 elements:
$
t_{25} = D_0 I_{12} F_U I_{12} D_0 I_{12} F_U I_{12} D_0 I_{12} \\ D_0 F_U I_{12}  D_0 I_{12}  D_0 F_U I_{12} D_0 I_{12} D_0 F_U D_0 F_U I_{12} D_0.
$

Constructing this set consisting of $n$ transvections according to Theorem 2.1 from \cite{Humph} means that we have constructed a generating set. Furthermore, as was studied in \cite{VSEM} a minimal generating set from transvections from $SL(n,\mathbb{Z})$ is of size $n$. This set of generators allows us to express the permutation matrix $P_5$ in the form
\begin{equation*}
    P_5 = F_U \cdot D_1 \cdot F_U \cdot D_1 =
    \begin{pmatrix}
        0 & 0 & 1 & 0 & 0 \\
        0 & 0 & 0 & 1 & 0 \\
        0 & 0 & 0 & 0 & 1 \\
        1 & 0 & 0 & 0 & 0 \\
        0 & 1 & 0 & 0 & 0
    \end{pmatrix}. \end{equation*} 

The minimality of this set is based on Lemma \ref{minimal3gen}.   
\end{proof}

\begin{remark}\label{FL}
The transformation from the previous generating set $\langle F_U, -D_0, I_{12} \rangle$ can be given by the formula: \begin{equation*}
    F_L = D_0 F_U D_0
\end{equation*}
proving that the set $F_L, D_0, I_{12}$ is also the generating set as well as $F_U, D_0, I_{12}$, because we made only equivalent invertible Nielsen transformations of generators.
Thus, instead $F_U$ we can use the matrix $F_L$.
\end{remark}

Consider the following involutions:
\begin{center}
    $
        I_{23} = \begin{pmatrix}
            1 & 0 & 0 & 0 & 0  \\
            0 & -1 & 1 & 0 & 0  \\
            0 & 0  & 1 & 0 & 0  \\
            0 & 0  & 0 & 1 & 0  \\
            0 & 0  & 0 & 0 & -1
        \end{pmatrix},
        I_{43} = \begin{pmatrix}
            1 & 0 & 0 & 0 & 0 \\
            0 & 1 & 0 & 0 & 0 \\
            0 & 0 & 1 & 0 & 0 \\
            0 & 0 & 1 & -1& 0 \\
            0 & 0 & 0 & 0 & 1
        \end{pmatrix}
            $ and $    I_{23} I_{43} =
    \begin{pmatrix}
        1 & 0 & 0 & 0 & 0 \\
        0 & -1 & 1 & 0 & 0 \\
        0 & 0 & 1 & 0 & 0 \\
        0 & 0 & 1 & -1& 0 \\
        0 & 0 & 0 & 0 & 1
    \end{pmatrix} =
    I_{43} I_{23}.
$
\end{center}

\textbf{Theorem.} \label{Maztriple}
    The set of involutions $\left\langle I_{23, 43}, D_0, F_L \right\rangle $ with two commuting involutions $I_{23, 43}, D_0$ is Mazurov triple \cite{Maz}  generates $ESL_3[\mathbb{Z}]$. 

Consider this set of generators $F_L, D_0, I_{23, 43}$, where
\begin{equation*}
    F_L = \begin{pmatrix}
            1 & 0 & 0 & 0 & 0 \\
            0 & 0 & 0 & 0 & 1 \\
            0 & 0 & 0 & 1 & 0 \\
            0 & 0 & 1 & 0 & 0 \\
            0 & 1 & 0 & 0 & 0
        \end{pmatrix}, \quad
    D_0 = \begin{pmatrix}
            0  & 0  & 0  & 0  & -1 \\
            0  & 0  & 0  & -1 & 0  \\
            0  & 0  & -1 & 0  & 0  \\
            0  & -1 & 0  & 0  & 0  \\
            -1 & 0  & 0  & 0  & 0
        \end{pmatrix}. 
    \end{equation*}

   \begin{proof}
       Proof. Since $[I_{23}, I_{43}]=E$, their product $I_{23} I_{43}= I_{23, 43}$ is also an involution. The Nielsen transformation $(D_0 F_U) I_{12} ( D_0 F_U)^{-1} = (D_0 F_U) I_{12} (F_U D_0 )=I_{23}$ lead us to generic set $\left\langle I_{23}, D_0, F_L \right\rangle $. In view of Remark \ref{FL} $F_L = D_0 F_U D_0$ that entails Nielsen transformation $F_U = D_0 F_L D_0$ to generic set $\left\langle I_{23, 43}, F_U, D_0 \right\rangle$. Since $I_{32} = F_U  I_{23} F_U$ as well as $I_{43} = D_0  I_{23} D_0 $ also, $I_{23} D_0 I_{23} D_0= I_{23, 43}$, we make an equivalent transformation of Nielsen to the new generating set $\left\langle I_{23, 43}, D_0, F_L \right\rangle $. Taking into account that  $\left\langle I_{12}, D_0, F_L \right\rangle $ generates $ESL_5[\mathbb{Z}]$, we deduce that $\left\langle I_{23, 43}, D_0, F_L \right\rangle $ also generates it.
   \end{proof} 
    
 \textbf{Remark}  Moreover, we obtain new triple with\textit{\textbf{ two commuting involutions}}
$$ D^2_0=e, I^2_{23, 43}=e,  F_U,   [D_0, I_{23, 43}]=e,  I_{43} = D_0  I_{23} D_0.$$

\begin{proof}
Based on $I_{23} D_0 I_{23} D_0= I_{23, 43}$ one can verify that
$[D^2_0, I_{23, 43}]= D_0 (I_{23} D_0 I_{23} D_0) D_0(I_{23} D_0 I_{23} D_0)^{-1}  = D_0 e D_0 =e$, $D^2_0 =e$, $I_{23, 43}^2=e$.
   \end{proof}

Consider the following involutions:
\begin{center}
    $
        I_{23} = \begin{pmatrix}
            1 & 0 & 0 & 0 & 0  \\
            0 & -1 & 1 & 0 & 0  \\
            0 & 0  & 1 & 0 & 0  \\
            0 & 0  & 0 & 1 & 0  \\
            0 & 0  & 0 & 0 & -1
        \end{pmatrix},
        I_{43} = \begin{pmatrix}
            1 & 0 & 0 & 0 & 0 \\
            0 & 1 & 0 & 0 & 0 \\
            0 & 0 & 1 & 0 & 0 \\
            0 & 0 & 1 & -1& 0 \\
            0 & 0 & 0 & 0 & 1
        \end{pmatrix}
            $ and $    I_{23} I_{43} =
    \begin{pmatrix}
        1 & 0 & 0 & 0 & 0 \\
        0 & -1 & 1 & 0 & 0 \\
        0 & 0 & 1 & 0 & 0 \\
        0 & 0 & 1 & -1& 0 \\
        0 & 0 & 0 & 0 & 1
    \end{pmatrix} =
    I_{43} I_{23}.
$
\end{center}


Since involutions commute, their product $I_{23} I_{43}$ is also an involution. Consider the system of generators $F_L, D, I_{23, 43}$:
\begin{equation*}
    F_L = \begin{pmatrix}
            1 & 0 & 0 & 0 & 0 \\
            0 & 0 & 0 & 0 & 1 \\
            0 & 0 & 0 & 1 & 0 \\
            0 & 0 & 1 & 0 & 0 \\
            0 & 1 & 0 & 0 & 0
        \end{pmatrix}, \quad
    D_0 = \begin{pmatrix}
            0  & 0  & 0  & 0  & -1 \\
            0  & 0  & 0  & -1 & 0  \\
            0  & 0  & -1 & 0  & 0  \\
            0  & -1 & 0  & 0  & 0  \\
            -1 & 0  & 0  & 0  & 0
        \end{pmatrix}, \quad
    I_{23, 43} =
    \begin{pmatrix}
        1 & 0 & 0 & 0 & 0 \\
        0 & -1 & 1 & 0 & 0 \\
        0 & 0 & 1 & 0 & 0 \\
        0 & 0 & 1 & -1& 0 \\
        0 & 0 & 0 & 0 & 1
    \end{pmatrix}.
    \end{equation*}

    Proof. In view of Remark \ref{FL} $F_L = D_0 F_U D_0$, that entails Nielsen transformation $F_U = D_0 F_L D_0$ to generic set $\left\langle I_{23, 43}, F_U, D_0 \right\rangle$. Since $I_{32} = F_U  I_{23} F_U$ as well as $I_{43} = D_0  I_{23} D_0 $ also $I_{23} D_0 I_{23} D_0= I_{23, 43}$, then we make an equivalent transformation of Tits to new generating set $\left\langle I_{23, 43}, D_0, F \right\rangle $. Taking into account that  $\left\langle I_{12}, D_0, F_L \right\rangle $ generates $ESL_3[\mathbb{Z}]$ we deduce that $\left\langle I_{23, 43}, D_0, F_L \right\rangle $ generates it too.
    
    Moreover, we obtain the triple with\textit{\textbf{two commuting involutions}}
$$ D_0, I_{23, 43},  F_U,   [D_0, I_{23, 43}]=e,  I_{43} = D_0  I_{23} D_0.$$

Based on $I_{23} D_0 I_{23} D_0= I_{23, 43}$ one can verify that
$[D_0, I_{23, 43}]= D_0 (I_{23} D_0 I_{23} D_0) (I_{23} D_0 I_{23} D_0) D_0 = D_0 e D_0 =e$.
    
    Note that in the generic set $F_U,  I_{23, 32}, D_0$ another involutions commute again because of $[F_U,  I_{23, 32}]=e$.
    If we consider new involution $I_{12}$ then $D_0 I_{12}D_0 =I_{54}$ and let $I_{12, 54} = I_{12} \times I_{54}$, as a result we have similar commuting pair $[I_{12, 54}, D_0 ]=e$. Another triple $ I_{54}=F_L  I_{23} F_L $ then  commuting pair is $[ F_L,  I_{23, 54} ]=e $.

\bf{Corollary 2.} {\sl The minimal sets of involutive generators $ggi$, both for $ESL_5[\mathbb{Z}]$ and for $ESL_5(\mathbb{F}_p)$, which are $sggi$ \cite{Leem} and also contain the Mazurov triple $(2\times2, 2)$ \cite{Maz}, are as follows: $\langle D_0, F_U, I_{23, 43} \rangle$,
  Proof. The proof is based on directly verification that $[D_{0}, I_{23, 43}]=e$.}
  

\bf{Definition.}  {\sl By the \textit{extended group} of \textit{(upper) unitriangular matrices} $EUT_n({\mathbb{F}})$ (over a field $\mathbb{F}$) we mean the unitriangular group $UT_n(\mathbb{F})$ \cite{SusUTn, Gol} that admits not only 1 but also -1 on the diagonal.}

{\sl 
 Recall that a group is called a 
 \textit{string group generated by involutions} $\{ \rho_0, \rho_1, \ldots, \rho_{n-1} \}$ (sggi) if there exists such ordering of the involutions wherein $\rho_i\rho_j = \rho_j\rho_i$ for every $i, j \in \{0, \ldots , n - 1 \}$ such that $| i - j| > 1$ \cite{Leem}.}

Assume there is indexed tuple of matrices $\{\rho_i\}$ such, that for all $i, j$ if $|i - j| > 1$ then $\rho_i \rho_j = \rho_j \rho_i$ and such a group calls $sggi$ group according to [Leem].

  \bf{Definition 4.}  {\sl By the \textit{extended group} of \textit{(upper) unitriangular matrices} $EUT_n({\mathbb{F}})$ (over a field $\mathbb{F}$) we mean the unitriangular group $UT_n(\mathbb{F})$ \cite{SusUTn, Gol} that admits not only 1 but also -1 on the diagonal. 

{\bf Property.}
{\sl $EUT_3[\mathbb{Z}] \simeq \langle \rho_1, \rho_2, \rho_3, \rho_4 \rangle$ with the following generating set of involutions comply with commuting property [Leem].}
  
  We prove that $EUT_3[\mathbb{Z}] \simeq \langle \rho_1, \rho_2, \rho_3, \rho_4 \rangle $ is $sggi$ group \cite{Leem} with following involutive generating set:}
$$
\rho_0 = \begin{pmatrix} -1 & 0 & 0 \\ 0 & 1 & 0 \\ 0 & 0 & 1 \end{pmatrix}, \quad
\rho_1 = \begin{pmatrix} -1 & 1 & 0 \\ 0 & 1 & 0 \\ 0 & 0 & 1 \end{pmatrix}, \quad
\rho_2 = \begin{pmatrix} 1 & 0 & 0 \\ 0 & 1 & 1 \\ 0 & 0 & -1 \end{pmatrix}, \quad
\rho_3 = \begin{pmatrix} 1 & 0 & 0 \\ 0 & 1 & 0 \\ 0 & 0 & -1 \end{pmatrix},
$$
{\sl where  $[\rho_1, \rho_3] = e$, $[\rho_0, \rho_3] = e$, $[\rho_0, \rho_2] = e$, this defines the type of involution generating set  $(2 \times 2, 2 \times 2)$.}

{\bf Remark}. {\sl If we set $D_0= \rho_0$, $F_L=\rho_1$, $I_{23, 43}=\rho_2$, then this order of involutions justifies that $ESL_5[\mathbb{Z}]$ and $ESL_5(\mathbb{F}_p)$ are string groups generated by involutions (sggi) \cite{Leem}.}

\section{Minimal generating set with a fourth order element.}
{\sl
Let ${{I}_{11}}=\left( \begin{matrix} -1 & 1  \\
  \, 0 &  1  \\
   \end{matrix} \right)$ and $T_4=\left( \begin{matrix} 0 & -1  \\
   1 &  0  \\
   \end{matrix} \right)$. 

   \begin{prop}
   The set $S= \l  T_4, I_{11}, \rho \r $ generates $ESL_2(\mathbb{Z})$.    
   \end{prop}
    We show that $S$ is the generic set.
   Squaring $T_4^2$ we get $-E$. To express a transvection we consider the product 
 $I_{11}T_4 =t_{11}$ and also $T_4 I_{11} = \left( \begin{matrix} 0 & -1  \\
   -1 &   1  \\
   \end{matrix} \right)$ which be denoted by $t$. Applying $-E$ we get $-E t^{-1}_{}= \left( \begin{matrix} 0 & 1  \\
   1 &   -1  \\
   \end{matrix} \right)$ that is inverse to $t_{11}$. In view of group axioms existing of $t^{-1}_{11}$ in the group closure of ${I}_{11}, T_4, -E$  entails existing of $t_{11}$. 
   The rest of transvection can be reached by rotation $\rho t_{12}= \left( \begin{matrix} 0 & 1  \\
                      1 & 1  \\
\end{matrix} \right)$.  Conjugation of $t_{11}$ lead as to $\rho  t_{12} \rho =t_{21}$.
   As well known the transvections $t_{11}$ and $t_{22}$ generate group $S{{L}_{2}}\left( \mathbb{Z} \right)$ and presents of $\rho$ with $det(\rho)=-1$ extends this $S{{L}_{2}}\left( \mathbb{Z} \right)$ to $ES{{L}_{2}}\left( \mathbb{Z} \right)$. }






{\sl
\begin{prop}
 There are no three involutions in $ESL(n,\mathbb{Z})$, one of which commutes with the others, do not generate the group $ESL(n,\mathbb{Z})$.
\end{prop}
A number of matrices with the norm less or equal than $R$ be denoted by $N(R)$.
All words over such generating set $\langle A, B, C \rangle$ that can be constructed over an alphabet from two involutions $A, B$ ($[A, B]$) are represented by one of the next sequences either ${{\text{(}AB\text{)}}^{n}}$ or $B{{(AB)}^{n}}$ or ${{(AB)}^{n}}A$ or $B{{(AB)}^{n}}A$.
It is sufficient to describe the form of the periodic part of these words, viz., ${{(AB)}^{n}}$ – then it will be clear that not all the matrices from $ES{{L}_{2}}\left( \mathbb{Z} \right)$ can be represented in this way.

For the proof, a specific form ${{(AB)}^{n}}$ is needed.
An additional proof is based on the fact that the group $ES{{L}_{2}}\left( \mathbb{Z} \right)$ is a two-parameter family. Therefore, it is unlikely to be covered by one-parameter families such as the sequence $AB$.
If $AB$ has determinant ±1, then it is either similar to $\left( \begin{matrix}
   \pm 1 & 1  \\
   0 & \pm 1  \\
\end{matrix} \right)$  or diagonalizable, and in the first case at least one of the eigenvalues has modulus  $\ge 1$,
however over the ring $\mathbb{Z}$ in a diagonal matrix there only can be elements 1 or -1 in any arrangement. 
For a matrix with a Jordan block size of 2 by 2 e.v. are $1$ or $-1$, thence $N(R)=cR$. Since the $n$-th power of such Jordan block is equal to $B^n=\left( \begin{matrix}
1 & n \\
0 & 1 \\
\end{matrix} \right)$, exactly the first $R$ terms are contained in a ball of radius $R$. R powers 
Since the $n$-th power of the Jordan block is equal to A, exactly the first $R$ terms (in a series of power of $B$) are contained in a ball of radius $R$.
At the same time, the number of elements $M$ where $M \in ESL_2(\mathbb{Z})$ grows as a function $R \sqrt{R}$, because if we fix first column of $M$, it remains to choose two elements of the second column by $\sqrt{R}$ ways in accordance with the theorem about the prime numbers distribution in order to  Diophantine equation $x \alpha - y \beta =1$ to be solvable over $\mathbb{Z}$.
Note that solvability of $x \alpha - y \beta =1$ is equivalent to $( \alpha, \beta)=1$ and therefore the theorem about prime numbers distribution is applicable to estimating the solutions number. 


If the matrix $A \in ESL_2(\mathbb{Z})$ has the diagonal form $A=\left( \begin{matrix}
\alpha & 0 \\
0 & \alpha^{-1} \\
\end{matrix} \right)$, where $ \alpha \in \mathbb{C}$ and $| \alpha | >1$, $|\alpha^{-1} | <1$,
this yields exponential growth of a matrix $(AB)^n$ norm,
then in the series in its powers $A^n$ the following growth function of the number of matrices with norm no greater than $R$ holds: $N(R)=clog_{\alpha}R$, where constant $C$ appears due to equivalence transformation.

Therefore, ${{\left( AB \right)}^{n}}$ yields either linear or exponential growth. (It is also possible that both eigenvalues of $AB$  have modulus 1, and it is diagonalizable—but then it will simply have finite order.)

The number of $ESL_2(\mathbb{F}_p)$ elements with norm no grater than $R$ is $R^2$ because of relation on the determinant of these elements $N(R)$ decrease from $R^{k^2}$ to $R^2$. 
Consequently, the number $N(R)$ for elements generated by $\langle A, B, C \rangle$ which forms only four sequences (${{\text{(}AB\text{)}}^{n}}$, $B{{(AB)}^{n}}$, ${{(AB)}^{n}}A$, $B{{(AB)}^{n}}A$) is less than number elements of $ESL_2(\mathbb{F}_p)$ with norm less than or equal to $R$, that's accomplish the proof. 
}

\subsection{Minimal involutive generating set for $ESL_3(\mathbb F_2)$.}
 {\sl The minimal of transvections generating $SL(n, K)$ is n that was found in \cite{Humph}.
 Involutive generating sets of linear groups over $F_2$ is subject of interest of many authors \cite{Nuz}. 
 Here we find all minimal involutive generating sets, here are three of them: 
\begin{center}
    $
        I_{11} = \begin{pmatrix}
            1 & 0 & 0 \\
            0 & 1 & 1 \\
            0 & 0 & 1
        \end{pmatrix}, \quad
        I_{12} = \begin{pmatrix}
            1 & 0 & 0 \\
            1 & 1 & 0 \\
            1 & 0 & 1
        \end{pmatrix}, \quad
        I_{13} = \begin{pmatrix}
            1 & 1 & 0 \\
            0 & 1 & 0 \\
            0 & 1 & 1
        \end{pmatrix}.
    $ \\[6pt]
    $
        I_{21} = \begin{pmatrix}
            1 & 0 & 0 \\
            1 & 1 & 0 \\
            0 & 0 & 1
        \end{pmatrix} \quad
        I_{22} = \begin{pmatrix}
            1 & 0 & 1 \\
            0 & 1 & 1 \\
            0 & 0 & 1
        \end{pmatrix} \quad
        I_{23} = \begin{pmatrix}
            1 & 1 & 1 \\
            0 & 0 & 1 \\
            0 & 1 & 0
        \end{pmatrix}
    $ \\[6pt]
    $
        I_{31} = \begin{pmatrix}
            1 & 0 & 0 \\
            1 & 1 & 1 \\
            0 & 0 & 1
        \end{pmatrix} \quad
        I_{32} = \begin{pmatrix}
            1 & 0 & 1 \\
            0 & 1 & 1 \\
            0 & 0 & 1
        \end{pmatrix} \quad
        I_{33} = \begin{pmatrix}
            1 & 1 & 1 \\
            0 & 0 & 1 \\
            0 & 1 & 0
        \end{pmatrix}
    $
\end{center}
But over $\mathbb{F}_2$ we have $-1\cong 1 mod 2$ then $ESL(3, \mathbb F_2) = SL(3, \mathbb F_2)$.
    Since $ESL(3,F_2)$ has not 2-generated involutive set, thence these sets are minimal as well.

There 117 involutions in $SL(3, \mathbb F_2)$ and 48672 triples (combinations) of involutions.
Here we find minimal involutive generating set for $ESL(3, \mathbb F_3)$:
\begin{center}
 $        I_{31} = \begin{pmatrix}
            2 & 0 & 0 \\
            0 & 2 & 0 \\
            2 & 2 & 2
        \end{pmatrix}, \quad
        I_{32} = \begin{pmatrix}
            2 & 0 & 2 \\
            0 & 2 & 0 \\
            0 & 0 & 2
        \end{pmatrix}, \quad
        I_{33} = \begin{pmatrix}
            2 & 2 & 2 \\
            0 & 0 & 2 \\
            0 & 2 & 0
        \end{pmatrix}.     $
\end{center}
}

\subsubsection*{Commuting involutions}

\bf{Theorem 2.} {\sl The minimal set of involutive generators $F_L, D_0, I_{23, 43}$ as for $ESL_5[\mathbb{Z}]$ as well as for $ESL_5[\mathbb{F}_p]$ is Mazurov triple \cite{Maz} with a such order:} 
\begin{equation*}
I_{23, 43} =
    \begin{pmatrix}
        1 & 0 & 0 & 0 & 0 \\
        0 & -1 & 1 & 0 & 0 \\
        0 & 0 & 1 & 0 & 0 \\
        0 & 0 & 1 & -1& 0 \\
        0 & 0 & 0 & 0 & 1
    \end{pmatrix},
    \quad
    F_L = \begin{pmatrix}
            1 & 0 & 0 & 0 & 0 \\
            0 & 0 & 0 & 0 & 1 \\
            0 & 0 & 0 & 1 & 0 \\
            0 & 0 & 1 & 0 & 0 \\
            0 & 1 & 0 & 0 & 0
        \end{pmatrix}, \quad
    D_0 = \begin{pmatrix}
            0  & 0  & 0  & 0  & -1 \\
            0  & 0  & 0  & -1 & 0  \\
            0  & 0  & -1 & 0  & 0  \\
            0  & -1 & 0  & 0  & 0  \\
            -1 & 0  & 0  & 0  & 0
        \end{pmatrix}. 
    \end{equation*}
{\sl    Thus, we obtain triple with \textit{\textbf{two commuting involutions}}
$$ D^2_0=e, I^2_{23, 43}=e,  F^2_L=e, \, I_{43} = D_0  I_{23} D_0, \,  [D_0, I_{23, 43}]=e.$$

Consider the product of this involutions.
\begin{center}
    $
        I_{23} = \begin{pmatrix}
            1 & 0 & 0 & 0 & 0  \\
            0 & -1 & 1 & 0 & 0  \\
            0 & 0  & 1 & 0 & 0  \\
            0 & 0  & 0 & 1 & 0  \\
            0 & 0  & 0 & 0 & -1
        \end{pmatrix},
        I_{43} = \begin{pmatrix}
            1 & 0 & 0 & 0 & 0 \\
            0 & 1 & 0 & 0 & 0 \\
            0 & 0 & 1 & 0 & 0 \\
            0 & 0 & 1 & -1& 0 \\
            0 & 0 & 0 & 0 & 1
        \end{pmatrix}
            $ and $    I_{23} I_{43} =
    \begin{pmatrix}
        1 & 0 & 0 & 0 & 0 \\
        0 & -1 & 1 & 0 & 0 \\
        0 & 0 & 1 & 0 & 0 \\
        0 & 0 & 1 & -1& 0 \\
        0 & 0 & 0 & 0 & 1
    \end{pmatrix} =
    I_{43} I_{23}.
$
\end{center}




\section{The criterion of equation $X^2=A$ solvability in $S{{L}_{3}}\left[ \mathbb{Z} \right]$}
Let ${{\lambda }_{1}},\,\,{{\lambda }_{2}},\,{{\lambda }_{3}}$ be  e.v. of $A\in S{{L}_{3}}\left[ \mathbb{Z} \right]$ provided $trA=\lambda _{1}^{{}}+\lambda _{2}^{{}}+\lambda_{3}^{{}}=a$,
$b= {{\lambda }_{1}}{{\lambda }_{2}}+{{\lambda }_{1}}{{\lambda }_{3}}+{{\lambda }_{2}}{{\lambda }_{3}}$. Let ${{\chi }_{A}}\left( x \right)={{x}^{3}}-a{{x}^{2}}+bx-1$ denotes characteristic polynomial for $A$.
According to Lemma 1 \cite{SkuESL} if ${{B}^{2}}=A$, then ${{\mu }_{1}}=\sqrt{{{\lambda }_{1}}},\,\,{{\mu }_{2}}=\sqrt{{{\lambda }_{2}}},\,\,{{\mu }_{3}}=\sqrt{{{\lambda }_{3}}}$, where ${{\mu }_{1}},\,\,{{\mu }_{2}},\,\,{{\mu }_{3}}$  are  e.v. of $B$.
We introduce the following notations $q={{\mu }_{1}}{{\mu }_{2}}+{{\mu }_{1}}{{\mu }_{3}}+{{\mu }_{2}}{{\mu }_{3}}$,  $trB=p={{\mu }_{1}}+{{\mu }_{2}}+{{\mu }_{3}}$. Let ${{\chi }_{B}}\left( x \right)={{x}^{3}}-p{{x}^{2}}+qx-1$ be characteristic polynomial of $B$.

Let $a = trA$ and
$b= \left( \begin{matrix} a_{11} & a_{12}  \\
                      a_{21} & a_{22}  \\
\end{matrix} \right) + \left( \begin{matrix} a_{11} & a_{13}  \\
                      a_{21} & a_{23}  \\
\end{matrix} \right) + \left( \begin{matrix} a_{12} & a_{13}  \\
                      a_{22} & a_{23}  \\
\end{matrix} \right)  = {{\lambda }_{1}}{{\lambda }_{2}}+{{\lambda }_{1}}{{\lambda }_{3}}+{{\lambda }_{2}}{{\lambda }_{3}}$.

{\bf Theorem 1.}
{\ The square root $\sqrt[{}]{A}$ of a diagonizible matrix $A \in S{{L}_{3}}\left[ \mathbb{Z} \right]$ belongs (up to matrix similarity) to $ES{{L}_{3}}\left[ \mathbb{Z} \right]$ iff
\begin{equation}({{p}^{4}}-2a{{p}^{2}}-8p+{{a}^{2}}-4b) ({{p}^{4}}-2a{{p}^{2}}+8p+{{a}^{2}}+4b)=0	
\end{equation}

is solvable over $ \mathbb{Z}$.
Moreover, the equivalent condition is true
\begin{equation*}
\left.\begin{aligned}
{{q}^{2}}-2p &=b\in \mathbb{Z}   \\
{{p}^{2}}-2q &= a\in \mathbb{Z}.
\end{aligned} \right\}
\end{equation*}

If $A$ possess Jordan cell $J_A(\lambda): \ dim(J_A(\lambda)) \geqslant 2$ then 
an equivalent condition is that $\lambda$ be a square in $\mathbb{Z}$.}

If $p=2$ then each matrix  $A \in S{{L}_{3}}\left( \mathbb{F}_p \right)$ has square root $\sqrt[{}]{A}\in S{{L}_{3}}(\mathbb{F}_p)$.

Proof.
\begin{equation*}
{\left. \begin{aligned}
   tr(A)-2\left( {{\mu }_{1}}{{\mu }_{2}}+{{\mu }_{1}}{{\mu }_{3}}+{{\mu }_{2}}{{\mu }_{3}} \right)\in \mathbb{Z} \\ 
  {{\left( {{\mu }_{1}}{{\mu }_{2}}+{{\mu }_{1}}{{\mu }_{3}}+{{\mu }_{2}}{{\mu }_{3}} \right)}^{2}}-2\left( {{\mu }_{1}}+{{\mu }_{2}}+{{\mu }_{3}} \right)\in \mathbb{Z} \\ 
\end{aligned} \right\} },
\end{equation*}
then
\begin{equation*}
\left.\begin{aligned}  & {{\left( {{\mu }_{1}}{{\mu }_{2}}+{{\mu }_{1}}{{\mu }_{3}}+{{\mu }_{2}}{{\mu }_{3}} \right)}^{2}}-2\left( {{\mu }_{1}}+{{\mu }_{2}}+{{\mu }_{3}} \right)=b\in \mathbb{Z} \\ 
 & {{p}^{2}}-2q=a\in \mathbb{Z} \\ 
\end{aligned} \right\}.
\quad \text{}
\end{equation*}

When we expand the brackets, we get the product ${\mu }_{1}{\mu }_{2}{\mu }_{3}$ we take into account the sign of the determinant for the group $SL(3, \mathbb{Z})$ therefore ${\mu }_{1}{\mu }_{2}{\mu }_{3}=1$ but if we consider $ESL(3, \mathbb{Z})$ ${\mu }_{1}{\mu }_{2}{\mu }_{3}=-1$. 

Expressing condition on coefficient of resultant and ${{\chi }_{A}}\left( x \right)$ we avail of them as follows here

\begin{equation*}
\left.\begin{aligned}
{{q}^{2}}-2p &=b\in \mathbb{Z}   \\
{{p}^{2}}-2q &= a\in \mathbb{Z}
\end{aligned} \right\}.
\quad \text{}
\end{equation*}

Let $2q=-b+{{p}^{2}}$, thus we eliminate the variable $q=\frac{{{p}^{2}}-a}{2}$. Its square ${{q}^{2}}=\frac{{{a}^{2}}-2a{{p}^{2}}+{{p}^{4}}}{4}$.

The substitution $p-\frac{a}{2}$ cancels the cubic term in ${{q}^{2}}=\frac{{{a}^{2}}-2a{{p}^{2}}+{{p}^{4}}}{4}$, since ${{\left( p-\frac{a}{2} \right)}^{4}}={{p}^{4}}-4\frac{a}{2}{{p}^{3}}+6\frac{{{a}^{2}}}{4}{{p}^{2}}-4\frac{{{a}^{3}}}{8}p+\frac{{{a}^{4}}}{{{2}^{4}}}$.
This entails 

$${{p}^{4}}-2a{{p}^{2}}-8p+{{a}^{2}}-4b=0.$$ 
\begin{cor}
Let $A \in SL_3(\mathbb{Z})$ provided $\det(A) = 1$. The square root of $A$ belongs to $ESL_3(\mathbb{Z})$ iff there exists an integer $p \in \mathbb{Z}$ such that the following equation holds:
    $$(p^4 - 2ap^2 - 8p + a^2 - 4b)(p^4 - 2ap^2 + 8p + a^2 + 4b) = 0,$$ 
    where $p$ corresponds to the trace of the prospective square root matrix $B$.
   \end{cor}
\begin{proof}
Note that since $A = B^2$, we have $\det(A) = (\det(B))^2$. Therefore, a necessary condition for any solution in $ ESL_3(\mathbb{Z})$ remains $\det(A) = 1$. If we assume $\det(B) = \mu_1 \mu_2 \mu_3 = -1$, the relation for the parameter $a = tr(A)$ remains unchanged: $$a = p^2 - 2q \implies q = \frac{p^2 - a}{2}$$However, the expansion for $b$ (the sum of principal minors of $A$) changes sign in the linear term due to the negative product of the eigenvalues: 
$$b = (\mu_1\mu_2 + \mu_1\mu_3 + \mu_2\mu_3)^2 - 2(\mu_1\mu_2\mu_3)(\mu_1 + \mu_2 + \mu_3)$$ $$b = q^2 - 2(-1)p = q^2 + 2p$$Substituting $q = \frac{p^2 - a}{2}$ into this modified equation yields:$$b = \left( \frac{p^2 - a}{2} \right)^2 + 2p$$$$4b = p^4 - 2ap^2 + a^2 + 8p$$Rearranging the terms, we obtain the characteristic condition for the trace $p$ when $\det(B) = -1$:$$p^4 - 2ap^2 + 8p + a^2 - 4b = 0$$
Combining this with the result from previous Theorem 1, we complete the proof the generalized criterion.   
\end{proof}

{\bf Remark 1.}
{\ 
 The square root of diagonizible matrix $A$ belongs (up to similarity) to $S{{L}_{3}}\left[ \mathbb{F}_p \right]$ iff
\begin{equation}{{p}^{4}}-2a{{p}^{2}}-8p+{{a}^{2}}-4b=0					
\end{equation}
is solvable over $p\in \mathbb{F}_p$.

The equivalent system of conditions
\begin{equation*}
\left.\begin{aligned}
 {{q}^{2}-2p} & = b  \in {\mathbb{F}_p}   \\
 {{p}^{2}-2q} & = a  \in {\mathbb{F}_p}
\end{aligned} \right\}
\end{equation*}
holds.

If $A$ possess Jordan cell $J_A(\lambda): \ dim(J_A(\lambda)) \geqslant 2$ then 
an equivalent condition is that $\lambda$ be a square in $\mathbb{Z}$.

If $p=2$ then each matrix  $A \in S{{L}_{3}}\left( \mathbb{F}_p \right)$ has square root $\sqrt[{}]{A}\in S{{L}_{3}}(\mathbb{F}_p)$.
}

\subsection{The necessary condition and criterion of equation ${{X}^{2}}=A$ solvability in $ S{{L}_{3}}\left[ \mathbb{Z} \right]$}

Let $A\in S{{L}_{3}}\left[ \mathbb{Z} \right]$ with e.v. ${{\lambda }_{1}},\,\,{{\lambda }_{2}},\,{{\lambda }_{3}}$ and assume that exists $B\in S{{L}_{3}}\left[ \mathbb{Z} \right]$ such that ${B}^{2}=A$. 

\begin{thm} \label{criterionSL(Z)}
If $A$ possess diagonal structure then a necessary condition for the solvability of an equation in ${{X}^{2}}=A$ in $S{{L}_{3}}\left[ \mathbb{Z} \right]$ up to a similarity transformation in $S{{L}_{3}}\left[ \mathbb{Z} \right]$.
\begin{equation}\label{trA}              tr(A)+2\left[ {{\mu }_{1}}{{\mu }_{2}}+{{\mu }_{1}}{{\mu }_{3}}+{{\mu }_{2}}{{\mu }_{3}} \right] \,\, \mathbf{is \,\, square \,\, in } \,  \mathbb{Z},
\end{equation}      
where ${{\mu }_{1}},\,\,{{\mu }_{2}},\,\,{{\mu }_{3}}$ are eigenvalues of the matrix $B$. Another words over $\mathbb{Z}$ we can either obtain a direct root of the equation $B^2=A$ or a matrix $B$, which is a square root of $A$ up to similarity. That is, $B^2$ is similar to $A$ over $Q$. 

 \end{thm} 
\begin{proof}
We will prove that if the $A$ is diagonalizable, and solutions (roots) of the equation ${{X}^{2}}=A$ in $S{{L}_{3}}\left[ \mathbb{Z} \right]$ exist then condition \eqref{trA} is satisfied.

Since $trA=tr{{B}^{2}}$ and $tr^{2}{{(B)}}={{\left( {{\mu }_{1}}+{{\mu }_{2}}+{{\mu }_{3}} \right)}^{2}}$, where ${{\mu }_{i}}$ are e.v. of $B$, then the equality $tr(A)=tr{{\left( B \right)}^{2}}-2\left[ {{\mu }_{1}}{{\mu }_{2}}+{{\mu }_{1}}{{\mu }_{3}}+{{\mu }_{2}}{{\mu }_{3}} \right]$ holds.  This equation can be brought into the form 
$$tr(A)+2\left[ {{\mu }_{1}}{{\mu }_{2}}+{{\mu }_{1}}{{\mu }_{3}}+{{\mu }_{2}}{{\mu }_{3}} \right]=tr^{2}{{\left( B \right)}}.$$ 
But $tr^{2}{{\left( B \right)}}={{\left( {{\mu }_{1}}+{{\mu }_{2}}+{{\mu }_{3}} \right)}^{2}}$ is square over $\mathbb{Z}$ therefore $tr(A)+2\left[ {{\mu }_{1}}{{\mu }_{2}}+{{\mu }_{1}}{{\mu }_{3}}+{{\mu }_{2}}{{\mu }_{3}} \right]$ must be square too.
Which accomplish the proof. 
\end{proof}
\begin{cor}
If $A$ possess diagonal structure then a necessary condition for the solvability of an equation in ${{X}^{2}}=A$ in $S{{L}_{3}}\left[ \mathbb{Z} \right]$ up to a similarity transformation in $S{{L}_{3}}\left[ \mathbb{Q} \right]$. 
\end{cor} 
\begin{equation}\label{trA}              tr(A)+2\left[ {{\mu }_{1}}{{\mu }_{2}}+{{\mu }_{1}}{{\mu }_{3}}+{{\mu }_{2}}{{\mu }_{3}} \right] \,\, \mathbf{is \,\, square \,\, in } \,  \mathbb{Z},
\end{equation}      
where ${{\mu }_{1}},\,\,{{\mu }_{2}},\,\,{{\mu }_{3}}$ are eigenvalues of the matrix $B$. Another words over $\mathbb{Z}$ we can either obtain a direct root of the equation $B^2=A$ or a matrix $B$, which is a square root of $A$ up to similarity. That is, $B^2$ is similar to $A$ over $Q$. 


We will prove that if the Jordan form ${{J}_{A}}$ is diagonalizable, then when condition \ref{trA} is satisfied, solutions (roots) of the equation ${{X}^{2}}=A$ in $S{{L}_{3}}\left[ \mathbb{Z} \right]$ exist, indeed one of the roots of $A$ will be a diagonal matrix with property $tr^{2}{{(B)}}={{\left( \sum\limits_{i=1}^{3}{\sqrt[{}]{{{\lambda }_{i}}}} \right)}^{2}}$ possessing the  e. v. $\sqrt[{}]{{{\lambda }_{i}}}$, $i=\overline{1,...,3}$. Its square is a matrix $A'$ similar to the matrix $A$, in other words ${{(B')}^{2}}=A'\sim A$.

Let $B$ be the square root of $A$ reduced to Jordan form. To $A$ be a square of $B = \left( \begin{matrix}
   \sqrt{\lambda}, & x, &0  \\
    0,  &\sqrt{\lambda}, &0 \\
    0,  &0 , &1
\end{matrix} \right)$ the equation $2\sqrt{\lambda }\,\cdot x=1$ has solution in $\mathbb{Z}$, that is possible iff $\sqrt{\lambda }\in \mathbb{Z}$, (but it is impossible over $\mathbb{Z}$, therefore there are not roots in this case). It remains to exclude existence of similar matrix to $A$ that is solution of equation. For this goal we reduce $A$ to canonical jordan form. Let ${{J}_{A}}=\left( \begin{matrix}
   1 & 1  \\
   0 & 1  \\
\end{matrix} \right)$ in order to matrix be similar to ${{J}_{A}}$ over $\mathbb{Z}$ we need to find non-degenerated matrix $U$ such that ${{J}_{A}}=UA'{{U}^{-1}}$. It is necessary to $A'$ has the same e.v. $\lambda =1$ (or  ). Applying the equivalent transformation to both part of ${{J}_{A}}=UA'{{U}^{-1}}$
Applying an equivalent transformation to both part of ${{J}_{A}}=UA'{{U}^{-1}}$ that preserves the similarity relation and consists of subtracting the diagonal matrix $\lambda E$ from both matrices, we obtain  ${{J}_{A}}-\lambda E=U(A'-\lambda E){{U}^{-1}}$. Since $\lambda =1$ it takes form ${{J}_{A}}-E=U(A'-E){{U}^{-1}}=UA'{{U}^{-1}}-UE{{U}^{-1}}$. Note that last matrix $UA'{{U}^{-1}}-UE{{U}^{-1}}=\left( \begin{matrix}
   0 & {{a}_{12}}  \\
   0 & 0  \\
\end{matrix} \right)$ is degenerates to one non zero number in ${{a}_{12}}$. In the left part of  this equation is the same situation ${{J}_{A}}-E=\left( \begin{matrix}
   0 & 1  \\
   0 & 0  \\
\end{matrix} \right)$. Therefore element ${{a}_{12}}\in \mathbb{Z}$ have to be associated elements with 1. 

But in $\mathbb{Z}$ such elements only $-1$ and 1. 
But for these elements there not similar matrix $A'$ over $\mathbb{Z}$. Thus there are not similar $A'$ to ${{J}_{A}}$ in $S{{L}_{2}}\left( \mathbb{Z}  \right)$ in contrast to the existence of a similarity transformation over the field $Q$.

Furthermore over ${{\text{F}}_{p}}$ a value of the symmetric polynomial ${{\mu }_{1}}{{\mu }_{2}}+{{\mu }_{1}}{{\mu }_{3}}+{{\mu }_{2}}{{\mu }_{3}}\in {{\text{F}}_{p}}$. Also due to the Lemma 1 we can obtain ${{\mu }_{i}}=\sqrt[{}]{{{\lambda }_{i}}}$. The symmetric polynomial ${{\mu }_{1}}{{\mu }_{2}}+{{\mu }_{1}}{{\mu }_{3}}+{{\mu }_{2}}{{\mu }_{3}}\in {{\text{F}}_{p}}$ be denoted by $P$.

Moreover since $tr{{\left( B \right)}^{2}}-P=\frac{tr{{\left( B \right)}^{2}}+tr\left( A \right)}{2}$ then $P=tr{{\left( B \right)}^{2}}-\frac{tr{{\left( B \right)}^{2}}+tr\left( A \right)}{2}=\frac{tr{{\left( B \right)}^{2}}-tr\left( A \right)}{2}$.
And finally, the last value on the right side, we can express as follows $tr^{2}{(B)}={{\left( \sum\limits_{i=1}^{3}{\sqrt[{}]{{{\lambda }_{i}}}} \right)}^{2}}$.

\begin{rem}
If characteristic polynomial of $A$ decompose in linear factors and $A$ has Jordan block of size 2, provided $\sqrt{\lambda} \in \mathbb{Q}$ then the equation $X^2 = A$ has solution in $SL(2, \mathbb{Q})$.    
\end{rem}  

In view of $A$ decomposes in linear factors over $\mathbb{Q}$ therefore $A$ possess Jordan block of size 2.
Let $B=\sqrt{A}$. Then it is not diagonalizable, so $B$ is reducible to Jordan form with Jordan block of size 2. Let us prove that this is as necessary condition and that $\sqrt {\lambda} \in \mathbb{Q}$ as a sufficient one, $B$ can be transformed to form
$B' = \left( \begin{matrix}
   \sqrt{\lambda}, &  1   \\
    0,  &\sqrt{\lambda}  \\
      \end{matrix} \right).$
Then $ (B')^2=A' =\left(\begin{matrix}
   \lambda & 2 \sqrt{\lambda}  \\
   0 & \lambda  \\
\end{matrix} \right)$. In order to prove that $A'$ be similar to ${{J}_{A}}$ over $\mathbb{Q}$ we need to find non-degenerated matrix $U$ such that ${{J}_{A}}=UA'{{U}^{-1}}$. It necessary entails $B^2=A'$. It is necessary to $A'$ has the same e.v. $\lambda$. 
Applying an equivalent transformation to both part of ${{J}_{A}}=UA'{{U}^{-1}}$ (that preserves the similarity relation) consists of subtracting the diagonal matrix $\lambda E$ from both matrices, we obtain  ${{J}_{A}}-\lambda E=U(A'-\lambda E){{U}^{-1}}$. Hence it takes form ${{J}_{A}}-E=U(A'-\lambda E){{U}^{-1}}=UA'{{U}^{-1}}-U\lambda E{{U}^{-1}}$. Note that the last matrix $UA'{{U}^{-1}}-U\lambda E{{U}^{-1}}=\left( \begin{matrix}
   0 & {{a}_{12}}  \\
   0 & 0  \\
\end{matrix} \right)$ degenerates to one non-zero number ${{a}_{12}}=2\sqrt{\lambda}$. In the left part of the equation above is the same situation ${{J}_{A}}-\lambda E=\left( \begin{matrix}
   0 & 1  \\
   0 & 0  \\
\end{matrix} \right)$. Therefore element ${{a}_{12}}\in \mathbb{Q}$ have to be associated elements with 1, which is possible iff  $\sqrt{\lambda} \in \mathbb{Q}$.

\begin{rem}
    In case if ${{J}_{A}}$ possess jordan cell of size 2 with e.v. $\lambda $ then $\sqrt{A}\in S{{L}_{3}}\left[ \mathbb{Q} \right]$ iff is true and $\lambda$ is square in $\mathbb{Q}$.     
\end{rem}
Proof. Let $B$ be the square root of $A$ reduced to Jordan form. To $A$ be a square of $B = \left( \begin{matrix}
   \sqrt{\lambda}, & x, &0  \\
    0,  &\sqrt{\lambda}, &0 \\
    0,  &0 , &1
\end{matrix} \right)$ the equation $2\sqrt{\lambda }\,\cdot x=1$ has solution in $\mathbb{Q}$, that is possible iff $\sqrt{\lambda }\in \mathbb{Q}$, (but it is impossible over $\mathbb{Z}$, therefore there are not roots in this case). It remains to exclude existence of similar matrix to $A$ that is solution of equation. For this goal we reduce $A$ to canonical jordan form. Let ${{J}_{A}}=\left( \begin{matrix}
   1 & 1  \\
   0 & 1  \\
\end{matrix} \right)$  to matrix be similar to ${{J}_{A}}$ over $\mathbb{Z}$ we need to find non-degenerated matrix $U$ such that ${{J}_{A}}=UA'{{U}^{-1}}$. It is necessary to$A'$ has the same e.v. $\lambda =1$ (or  ). Applying the equivalent transformation to both part of ${{J}_{A}}=UA'{{U}^{-1}}$
Applying an equivalent transformation to both part of ${{J}_{A}}=UA'{{U}^{-1}}$ that preserves the similarity relation and consists of subtracting the diagonal matrix $\lambda E$ from both matrices, we obtain  ${{J}_{A}}-\lambda E=U(A'-\lambda E){{U}^{-1}}$. Since $\lambda =1$ it takes form ${{J}_{A}}-E=U(A'-E){{U}^{-1}}=UA'{{U}^{-1}}-UE{{U}^{-1}}$. Note that last matrix $UA'{{U}^{-1}}-UE{{U}^{-1}}=\left( \begin{matrix}
   0 & {{a}_{12}}  \\
   0 & 0  \\
\end{matrix} \right)$ is degenerates to one non zero number in ${{a}_{12}}$. In the left part of  this equation is the same situation ${{J}_{A}}-E=\left( \begin{matrix}
   0 & 1  \\
   0 & 0  \\
\end{matrix} \right)$. Therefore element ${{a}_{12}}\in \mathbb{Q}$ have to be associated elements with 1. 

\begin{rem}
    If $A \in SL_{2} \mathbb{(Q)}$ or $A \in SL_{2} \mathbb{(F}_p)$ possess a jordan block ${{J}_{A}}$ of size 2 with e.v. $\lambda $ then $\sqrt{A}\in S{{L}_{2}}\left[ {{\text{Q}}} \right]$ ($\sqrt{A}\in S{{L}_{2}}\left[ {{\text{F}}_{p}} \right]$) iff $\sqrt{\lambda }\in {{\text{Q}}}$ or $\sqrt{\lambda }\in {{\text{F}}_{p}}$ respectively.
\end{rem}
\textbf{Proof}.
In view of $A$ decomposes in linear factors over $\mathbb{Q}$ therefore $A$ possess Jordan block of size 2.

Let $B=\sqrt{A}$ then it is not diagonalizable so $B$ is reducible to Jordan form with Jordan block of size 2. Considering this as a necessary condition and that $\sqrt {\lambda} \in \mathbb{Q}$ as a sufficient one, $B$ can be transformed to form
$B' = \left( \begin{matrix}
   \sqrt{\lambda}, &  1   \\
    0,  &\sqrt{\lambda}  \\
      \end{matrix} \right)$
then $ (B')^2=A' =\left(\begin{matrix}
   \lambda & 2 \sqrt{\lambda}  \\
   0 & \lambda  \\
\end{matrix} \right)$.

Actually to find solution $B$ in Jordan form with Jordan cell of size 2 such that $J^2_B= \left( \begin{matrix}
   {\lambda}, & 2\sqrt{\lambda},   \\
    0,  & {\lambda} \\
        \end{matrix} \right)$ the equation $2\sqrt{\lambda }\,\cdot x=1$
        or similarity transformation has to exists over $\mathbb{Q}$ i. e. non-generated matrix $U \in S{{L}_{2}}\left[ {{\text{Q}}} \right]$ such that $U J_B U^{-1}=J_A$.

Since $J_B^2=\left( \begin{matrix}
   {\lambda}, & 2\sqrt{\lambda},   \\
    0,  & {\lambda} \\
        \end{matrix} \right)$ then conjugation by $U= \left( \begin{matrix}
   \frac{1}{k} , & 0,   \\
    0,  & k \\
        \end{matrix} \right)$ can transforms it to form $A'= \left( \begin{matrix}
   {\lambda} , & \frac{2\sqrt{\lambda}}{k},   \\
    0,  & k \\
        \end{matrix} \right)$ which in case of $\sqrt{\lambda} \in \mathbb{Q}$ lead us to solution $J_A$ in jordan form in $SL_{2} \mathbb{(Q)}$.

In order to show absence of solutions if $\sqrt{\lambda} \notin \mathbb{Q}$ we apply equivalent transformation 
transformation to both part of ${{J}_{A}}=UA'{{U}^{-1}}$ that preserves the similarity relation and consists of subtracting the diagonal matrix $\lambda E$ from both matrices, we obtain  ${{J}_{A}}-\lambda E=U(A'-\lambda E){{U}^{-1}}$. Since $\lambda =1$ it takes form ${{J}_{A}}-E=U(A'-E){{U}^{-1}}=UA'{{U}^{-1}}-UE{{U}^{-1}}$. Note that last matrix $UA'{{U}^{-1}}-UE{{U}^{-1}}=\left( \begin{matrix}
   0 & {{a}_{12}}  \\
   0 & 0  \\
\end{matrix} \right)$ is degenerates to one non zero number in ${{a}_{12}}$. In the left part of  this equation is the same situation ${{J}_{A}}-E=\left( \begin{matrix}
   0 & 1  \\
   0 & 0  \\
\end{matrix} \right)$.

Note that the last matrix $UA'{{U}^{-1}}-U\lambda E{{U}^{-1}}=\left( \begin{matrix}
   0 & {{a}_{12}}  \\
   0 & 0  \\
\end{matrix} \right)$ degenerates to one non-zero number ${{a}_{12}}=2\sqrt{\lambda}$. In the left part of the equation above is the same situation ${{J}_{A}}-\lambda E=\left( \begin{matrix}
   0 & 1  \\
   0 & 0  \\
\end{matrix} \right)$. 
Therefore element ${{a}_{12}}\in \mathbb{Q}$ have to be associated elements with 1, which is possible iff $\sqrt{\lambda} \in \mathbb{Q}$.

        the equation $2\sqrt{\lambda }\,\cdot x=1$ have to be solvable in ${{\text{F}}_{p}}$ or equivalently $\sqrt{\lambda }\in {{\text{F}}_{p}}$, by the statements about similarity transformation over a field always exists a matrix $U$ such that $U A U^{-1}= J_A$, so if it occurs that $B^2=A =  \left( \begin{matrix}
   {\lambda}, & 2k, &0  \\
    0,  & {\lambda}, &0 \\
    0,  &0 , &1  
    \end{matrix} \right)$ 
(but it is impossible over $\mathbb{Z}$, therefore there are not roots in that case).

\section{Matrix roots of higher powers}

\textbf{Proposition}. Let $A, B \in G{{L}_{2}}({{\mathbb{F}}_{p}})$ if $ B$ is root of equation ${{X}^{3}}=A$, then

$$B=\frac{A+\tr(\sqrt[3]{A})\sqrt[3]{\det (A)}} {\left( \tr{ \sqrt[3]{A} }\right)^{2}-\sqrt[3]{\det (A)}},$$
 in the case $A\in S{{L}_{2}}({{\mathbb{F}}_{p}})$ we specify the values of the formula parameters taking into account that $\det (A)=1$.



\begin{proof}
 If $\sqrt[3]{A}\in S{{L}_{2}}({{\mathbb{F}}_{p}})$ then we consider Cayley-Hamilton equation (C.H.E.) ${{A}^{3}}-\tr\left( A \right){{A}^{2}}+\left( {{\lambda }_{1}}{{\lambda }_{2}}+{{\lambda }_{1}}{{\lambda }_{3}}+{{\lambda }_{2}}{{\lambda }_{3}} \right)A-\det \left( A \right)=0$.
Note, that $\tr{{\left( A \right)}^{2}}={{\left( {{\lambda }_{1}}+{{\lambda }_{2}}+{{\lambda }_{3}} \right)}^{2}}=\lambda _{1}^{2}+\lambda _{2}^{2}+\lambda _{3}^{2}-\left( {{\lambda }_{1}}{{\lambda }_{2}}+{{\lambda }_{1}}{{\lambda }_{3}}+{{\lambda }_{2}}{{\lambda }_{3}} \right)$.

Consider  C.H.E.  for  $A:\, dimA=2$,  ${{A}^{2}}-\tr\left( A \right)\cdot A+\det \left( A \right)\cdot 1=0$.
Multiplying last equation on $A$ admit us obtain the chain of transformations:
\begin{align} \label{KHE}
   {{A}^{3}} & =\left( \tr \left( A \right)A-\det \left( A \right) \right)A=\tr\left( A \right){{A}^{2}}- \det \left( A \right)A= \nonumber\\
   & =\tr\left( A \right)\left( \tr\left( A \right)A-\det \left( A \right) \right)-\det \left( A \right)A=  \nonumber\\[-3mm]
   & ~ \\[-3mm]
  & =\tr{{\left( A \right)}^{2}}A-\tr\left( A \right)\det \left( A \right)-\det \left( A \right)A= \nonumber\\
 & =\left( \tr{{\left( A \right)}^{2}}-\det \left( A \right) \right)A-\tr\left( A \right)\det \left( A \right). \nonumber
\end{align}

By applying substitute matrix $\sqrt[3]{A}$ instead of $A$ we express

\begin{equation}\label{2}
\sqrt[3]{A}=\frac{A + \tr\left(     \sqrt[3]{A} \right)\sqrt[3]{\det A}}{\tr^{2}{{\left( \sqrt[3]{A} \right)}}-\sqrt[3]{\det \left( A \right)}}.					 
\end{equation}

Thus, $\sqrt[3]{A}=\frac{A\,+\tr\left( \sqrt[3]{A} \right)\sqrt[3]{\det \left( A \right)}}{\,\left( t{{r}^{2}}\left( \sqrt[3]{A} \right)-\sqrt[3]{\det \left( A \right)} \right)}$.

Note that $\det \left( \sqrt[3]{A} \right)=\sqrt[3]{\det \left( A \right)}$ because of a determinant is homomorphism.

But $\tr\left( \sqrt[3]{A} \right)$ is still not computed.
 From \eqref{KHE}
 we conclude \\ ${A}^{3}=\left( \tr{{\left( A \right)}^{2}}-\det \left( A \right) \right)A-\tr\left( A \right)\det \left( A \right).$ Computing a trace from both sides we obtain
$\tr\left( {{A}^{3}} \right)=\tr{{\left( A \right)}^{3}}-3\det \left( A \right)\tr\left( A \right)$.

Substituting $\sqrt[3]{A}$ instead of $A$, we get
$\tr\left( A \right)=\tr{{\left( \sqrt[3]{A} \right)}^{3}}-3\sqrt[3]{\det A}\tr\left( \sqrt[3]{A} \right).$

Therefore we need to solve
$\tr\left( A \right)=\tr{{\left( \sqrt[3]{A} \right)}^{3}}-3\sqrt[3]{\det A}\tr\left( \sqrt[3]{A} \right)$.

We denote $\sqrt[3]{A}$ by $X$ and obtain the equation

$${{X}^{3}}-3\sqrt[3]{\det \left( A \right)}X-\tr\left( A \right)=0.$$

The \emph{solvability} of this equation in $SL_2({\F}_p)$
is equivalent to the \emph{existence} of a trace $\tr \sqrt[3]{A}$ in the base field ${\F}_p$. Recall that for a convenience we denote $\sqrt[3]{A}$ by $B$.
Therefore, we find $\tr B$ as a sum of cubic roots from solutions $\lambda_{1}^3$, $\lambda_{2}^3$ of the equation ${{x}^{2}}-(trA)x+\det A=0$, 
taking into account that $\tr B=\tr\sqrt[3]{A}={{\lambda }_{1}}+{{\lambda }_{2}}$ then $\lambda _{1}^{3}+\lambda _{2}^{3}=trA$.
 Indeed, each root ${{\lambda^3}_{i}}$ satisfies C.H.E. $${{x}^{2}}-(\tr  A)x+\det A=0$$ 
whose roots: $\lambda _{1,2}^{3}=\frac{trA\pm \sqrt{t{r^{2}}A-4\det A}}{2}$, therefore $$trB=\sqrt[3]{\frac{trA+\sqrt{t{{r}^{2}}A-4\det A}}{2}}\cdot \varepsilon +\sqrt[3]{\frac{trA-\sqrt{t{{r}^{2}}A-4\det A}}{2}}\cdot \bar{\varepsilon },$$ wherein $\varepsilon ={{e}^{\frac{2\pi i\cdot k}{3}}},\,\,k\in \mathbb{Z}$ i.e. $ \varepsilon =\cos \frac{2\pi k}{3}+i\sin \frac{2\pi k}{3},\,\,\,k\in \mathbb{Z}.$

Now we consider two singular cases:
\begin{itemize}
\item $(\ tr{B})^2 - \det{B}=0$, where $B= \sqrt[3]{A}$. \\
In this case from \eqref{2} we obtain $$A=B^3 =  - \tr{B} \det{B} \cdot E = -(\tr{B})^3 \cdot E.$$
From that we can compute $\tr{B}$ as a root of the equation $x^3+\dfrac{\tr{A}}{2}=0$.

\item If we consider additional matrix $2\times 2$, $B \notin SL_2({\F})$ and $B^3=0$, then it's minimal canceling polynomial is $X^2$ or $X$. By Celly Hamilton equation (C.H.E) $B^2 - tr{B} \cdot B + \det{B} \cdot I = 0$, which leads us to $tr{B}=0, \det{B}=0$.
\end{itemize}
\end{proof}

Let us define sequences $s_n = \tr{B} ~s_{n-1} + t_{n-1}$ and
 $t_n = - \det{B}~ s_{n-1}$ with initial conditions $s_1 = 1, t_1 = 0$, $s_2 = tr B$ and $t_2 = - det B$.
Now we prove the following Lemma.

{\bf Lemma.}
Sequences $s_n$, $t_n$ satisfy recurrent equation with characteristic polynomial $c(x)$ which is also characteristic polynomial for matrix $B$.

{\bf Theorem 2.}
 Let $n \geqslant 3$ and $A\in M_2(\mathbb{F}_p)$. If $A \not = c \cdot E$ for any $c \in \mathbb{F}_p$ and $R = \{B \in M_2(\mathbb{F}_p)  \mid B^n = A\}$  set of it's  $n$-th roots, then next inclusion follows:

\begin{gather*}
   R \subset \left\{B \in M_2(\mathbb{F}_p)  \left|   B=\dfrac{A+b~ Q_{n-2}(a,b) \cdot E}{Q_{n-1}(a,b)},~  b^n =\det{A} \right.,  ~P_n(a,b) = \tr{A} \right\}. \end{gather*}


\subsection{Recursive formula of $n$-th power root in the matrix ring $M_2(\mathbb{F}_p)$ }
Here and below we denote identity matrix from the matrix ring $M_2(\mathbb{F}_p)$ by $I$, in contrast to the identity matrix from the group $SL(2,\mathbb{Z})$ denoted as $E$.

\begin{prop} Let $A\in M_2(\mathbb{F}_p)$. Then it's cube roots $R = \{B \in M_2(\mathbb{F}_p)  \mid B^3 = A\}$ can be obtained as follows:
\begin{enumerate}
\item If $A=0$, then $R=\{B \in M_2(\mathbb{F}_p) \mid \det{B}=0,~\tr{B}=0 \}$;
\item If $A = c^3 I$, where $c \in \mathbb{F}_p /  \{0\} $, then $R=\{c \cdot B \in M_2(\mathbb{F}_p) \mid B^3 = I\}$;
\item In other cases $R \subset \left\{B \in M_2(\mathbb{F}_p) \left|  B=\dfrac{A+ab \cdot I}{a^2-b} \right.,~ b^3=\det{A}, ~a^3-3ab = \tr{A} \right\}$.
\end{enumerate}
\end{prop}

\begin{proof}
\begin{enumerate}
\item If $B^3=0$, then it's minimal canceling polynomial is $X^2$ or $X$. By Celly Hamilton equation (C.H.E) $B^2 - \tr{B} \cdot B + \det{B} \cdot I = 0$, which leads us to $\tr{B}=0, \det{B}=0$;
\item If $B$ is a solution of $X^3 - c^3 \cdot I=0$, then it's easy to see that $B' = c^{-1} B$ is a solution of $X^3 - I=0$;
\item Consider C.H.E for $B$:
$$B^2 - \tr{B} \cdot B + \det{B} \cdot I = 0.$$
Multiplying last equation by $B$ we proceed with the following chain of transformations:
\begin{multline*}
 B^3=
(\tr{B}\cdot B - \det{B} \cdot I) \cdot B =
\tr{B} \cdot B^2 - \det{B}\cdot B =
\tr{B} (\tr{B}\cdot B - \det{B} \cdot I) - \det{B} \cdot B=\\
=(\tr{B})^2 \cdot B - \tr{B} \det{B} \cdot I - \det{B} \cdot B =
((\tr{B})^2 - \det{B}) \cdot B - \tr{B} \det{B} \cdot I.
\end{multline*}
If $(\tr{B})^2 - \det{B}=0$, then we obtain $A=B^3 =  -\tr{B} \det{B} \cdot I =(-\tr{B})^3 \cdot I$, which leads us to previous cases.\\
\\
Otherwise $(\tr{B})^2 - \det{B}\not=0$ and we express $B$:
$$ B = \dfrac{B^3+\tr{B}\det{B}~I}{(\tr{B})^2-\det(B)}$$

Now since $B^3=A$ we conclude $\det{A} = \det{B^3} = (\det{B})^3$ and hence $\det{B}$ is a root of polynomial $x^3 - \det{A} = 0$.

Last thing remaining is to find $\tr{B}$.\\
By computing trace from both sides of $A=((\tr{B})^2 - \det{B}) \cdot B - \tr{B} \det{B} \cdot I$ we get:
$$\tr{A} = (\tr{B})^3 -3 \tr{B} \det{B}$$
From which we conclude that $\tr{B}$ is a root of $x^3-3 \det{B} \cdot x-\tr{A}=0$.

\end{enumerate}
\end{proof}

In general case we define complete symmetric polynomial of $n$-th degree in two variables: $$h_n(x,y) = \sum\limits_{k=0}^n x^ky^{n-k}.$$
In view of the fundamental theorem of symmetric polynomials $\exists !$ polynomial $Q(x,y) \in \mathbb{F}_p[x,y]$, such that:
$Q(e_1, e_2) = h_n$, where $e_1 = x+y$, $e_2 = xy$ --- elementary symmetric polynomials.

Likewise we determine the power symmetric polynomial of $n$-th degree in two variables: $$p_n(x,y) = x^n + y^n.$$
And polynomial  $P(x,y) \in \mathbb{F}_p[x,y]$, such that $P(e_1, e_2) = p_n$.

\begin{thm}

 Let $n \geqslant 3$ and $A\in M_2(\mathbb{F}_p)$. If $A \not = c \cdot I$ for any $c \in \mathbb{F}_p$ and $R = \{B \in M_2(\mathbb{F}_p)  \mid B^n = A\}$  set of it's  $n$-th roots, then next inclusion follows:
$$R \subset \left\{B \in M_2(\mathbb{F}_p) \left|  B=\dfrac{A+b~ Q_{n-2}(a,b) \cdot I}{Q_{n-1}(a,b)} \right.,~ b^n =\det{A}, ~P_n(a,b) = \tr{A} \right\}$$


\end{thm}

\begin{proof}

Let $B \in M_2(\mathbb{F}_p)$ be a root of equation $X^n = A$. Also consider it's C.H.E.\\ $c(X) = X^2 - \tr{B} X + \det{B} \cdot I$\\
\\
Then $X^n \underset{c(X)}  \equiv s_n X + t_n I$ for some $s_n, t_n \in \mathbb{F}_p$ and since $c(B) = 0$ we have
\begin{equation}\label{recformulaA} A = s_n B +t_n I. \end{equation}

Now we prove the following lemma

\begin{lem}\label{recpower}
Sequences $s_n$, $t_n$ satisfy recurrent equation with characteristic polynomial $c(x)$.

\end{lem}

\begin{proof}

$X^n = X \cdot X^{n-1} \underset{c(X)}  \equiv X \cdot (s_{n-1}X + t_{n-1} I) = s_{n-1}X^2 + t_{n-1} X \underset{c(X)}  \equiv s_{n-1}(\tr{B} X -\\- \det{B} \cdot I) + t_{n-1} X = (s_{n-1} \tr{B} + t_{n-1}) X - s_{n-1} \det{B} \cdot I$

Or by definition of $s_n$ and $t_n$:
\begin{equation}\label{rec} \begin{cases} s_n = \tr{B} ~s_{n-1} + t_{n-1}\\
t_n = - \det{B}~ s_{n-1}
\end{cases}\end{equation}

By summing up first expression multiplied by $\det{B}$ with the second one multiplied by $\tr{B}$ we get:
$$ \det{B}~ s_n + \tr{B}~ t_n = \det{B}~ t_{n-1} $$
or
$$ \det{B}~ s_n =  \det{B}~ t_{n-1} - \tr{B}~ t_n $$

Substituting into second equation of \eqref{rec} we obtain:
$$t_n-\tr{B}~t_{n-1}+\det{B}~t_{n-2}=0$$
Since $s_n $ and $t_n$ are linearly dependant it follows that $s_n$ satisfy the same recurrent.

\end{proof}


Since $X^1 \underset{c(X)}  \equiv X + 0 \cdot I$ and $X^2 \underset{c(X)}  \equiv \tr{B} X - \det{B} \cdot I$, we have $s_1 = 1, t_1 =0, s_2 = \tr{B}$ and $t_2 = - \det{B}$.

Consider algebraic closure of $\mathbb{F}_p$ that is $\widehat{\mathbb{F}_p}$. Let $\lambda_1, \lambda_2$ be roots of $c(x)$ in $\widehat{\mathbb{F}_p}$ (eigenvalues of B).

\begin{enumerate}

\item If $\lambda_1 \not=\lambda_2$ and $\lambda_1 \lambda_2 = \det{B} \not= 0$:

$$s_n = c_1 \lambda_1^n + c_2 \lambda_2^n, ~t_n = c_1' \lambda_1^n + c_2' \lambda_2^n$$

In cases $n= \overline {1,2}$ for $s_n$ we get:

$$\begin{cases} c_1 \lambda_1 + c_2 \lambda_2 = 1,\\
c_1 \lambda_1^2 + c_2 \lambda_2^2 = \tr{B}.
\end{cases}$$

Solving the system using Kramer's rule we obtain: $$c_1 = \dfrac{\lambda_2^2 - \lambda_2  \tr{B}}{\lambda_1 \lambda_2^2 - \lambda_1^2 \lambda_2} = -\dfrac{1 }{\lambda_2 - \lambda_1}, ~c_2 =  \dfrac{\lambda_1 \tr{B} - \lambda_2^2}{\lambda_1 \lambda_2^2 - \lambda_1^2 \lambda_2} =  \dfrac{1}{\lambda_2 - \lambda_1}$$

Substituting constants

\begin{equation} \label{s_n} s_n = \dfrac{\lambda_2^n - \lambda_1^n}{\lambda_2 - \lambda_1} = h_{n-1}(\lambda_1, \lambda_2). \end{equation}

In cases $n= \overline {1,2}$ for $t_n$ we get:

$$\begin{cases} c_1' \lambda_1 + c_2' \lambda_2 = 0,\\
c_1' \lambda_1^2 + c_2' \lambda_2^2 = - \det{B}
\end{cases}$$

Solving the system using Kramer's rule we obtain: $$c_1' = \dfrac{\lambda_2 \det{B}}{\lambda_1 \lambda_2^2 - \lambda_1^2 \lambda_2} = \dfrac{\lambda_2 }{\lambda_2 - \lambda_1}, ~c_2' = - \dfrac{\lambda_1 \det{B}}{\lambda_1 \lambda_2^2 - \lambda_1^2 \lambda_2} =  -\dfrac{\lambda_1 }{\lambda_2 - \lambda_1}$$

Similarly substituting constants in the second equation:

\begin{equation} \label{t_n}
t_n = \dfrac{\lambda_1^n\lambda_2-\lambda_1\lambda_2^n}{\lambda_2 - \lambda_1}= -\det{B} \cdot \dfrac{\lambda_1^{n-1}-\lambda_2^{n-1}}{\lambda_1-\lambda_2} = -\det{B}~ h_{n-2}(\lambda_1, \lambda_2) \end{equation}

\item In general case for each $n \ge 3$ we consider polynomial $D_n(\lambda_1, \lambda_2) = h_{n-1} - \tr{B} h_{n-2} + \det{B} h_{n-3}$. It's a continuous function of variables $\lambda_1, \lambda_2$.

Previously we proved that $D_n(\lambda_1, \lambda_2)=0$ if $\lambda_1 \not= \lambda_2$ and $\lambda_i \not = 0$.

From continuity follows that $D_n(\lambda_1, \lambda_2)=0$  for
 each $ \lambda_1, \lambda_2$ and hence formulas \eqref{s_n} and \eqref{t_n} are fulfilled for
 each $\lambda_1, \lambda_2$.

\end{enumerate}

Now that we have found $s_n$ and $t_n$ we return to equation \eqref{recformulaA}. If $s_n=0$, then $A = t_n I$ which contradicts conditions of the theorem. Dividing both sides by $s_n$ we get formula

\begin{equation}\label{recroot} B = \dfrac{A - t_n I}{s_n} = \dfrac{A + \det{B}~ h_{n-2}(\lambda_1, \lambda_2) \cdot I}{h_{n-1}(\lambda_1, \lambda_2)}=\dfrac{A + \det{B}~ Q_{n-2}(\tr{B}, \det{B}) \cdot I}{Q_{n-1}(\tr{B}, \det{B})}\end{equation}

The last thing remaining is to express $\det{B}$ and $\tr{B}$ in terms of $A$.

Since $\det{A} = \det{B^n} = \det{B}^n$, $\det{B}$ can be obtain as root of polynomial $x^n=\det{A}$.

To find $\tr{B}$ we compute trace from both sides of $B = \dfrac{A - t_n I}{s_n}$:
\begin{multline*} \tr{A} =  \tr{B}~ s_n +  2~ t_n = \tr{B}~ h_{n-1}(\lambda_1, \lambda_2) - 2 \det{B}~h_{n-2}(\lambda_1, \lambda_2) =\\=  h_{n}(\lambda_1, \lambda_2) -  \lambda_1\lambda_2~h_{n-2}(\lambda_1, \lambda_2) = \lambda_1^n+\lambda_2^n = p_n(\lambda_1, \lambda_2) = P_n(\tr{B}, \det{B}). \end{multline*}

\end{proof}

Let us check the recursive formula \eqref{recroot}, note that Lemma \eqref{recpower} entails equality ${{X}^{n}}=\left( {{s}_{n-1}}trB+{{t}_{n-1}} \right)X-{{s}_{n-1}}\det B\cdot I$ where the coefficients $s_n$ and  $t_n$ are determined by the system: 

$$\begin{cases} {{s}_{n}}={{s}_{n-1}}trB+{{t}_{n-1}} ,\\
{{t}_{n}}=-\det B\cdot {{s}_{n-1}},
\end{cases}$$

therefore, for root of 3-rd power $\sqrt[3]{A}$ can be expressed from the next equality
${{X}^{3}}=\left( {{s}_{2}}trB+{{t}_{2}} \right)X-{{s}_{2}}\det B\cdot I$,
which lead us to solution $X=\frac{{{X}^{3}}+{{s}_{2}}\det B\cdot I}{{{s}_{2}}trB+{{t}_{2}}}$.
Wherein ${{s}_{2}}=trB$, ${{t}_{2}}=-1=-\det B$, ${{t}_{1}}=0$.

Let 
${{{B}}}=\left( \begin{array}{rr}
   1\,\,\, &2 \\
  1\,\, & 3 \\
\end{array} \right)$ consequently $ B^4 = \left( \begin{array}{rr}
   41 \,\, &112 \\ 
  56 \,\,&153 \\ 
\end{array} \right)$   
then, by means of the recursive formula \eqref{recformulaA}
${{X}^{4}}=\left( {{s}_{3}}trB+{{t}_{3}} \right)X-{{s}_{3}}\det B\cdot I$ and ${{s}_{3}}={{s}_{2}}trB+{{t}_{2}}$, ${{t}_{3}}=-\det B{{s}_{2}}$
therein ${{s}_{2}}=trB=4$,  ${{s}_{3}}=4trB-4=15,\,\,{{t}_{4}}=-15$, ${{s}_{4}}=56$, ${{t}_{2}}=-1$, ${{t}_{1}}=0$ we successfully check our main formula \eqref{recroot} 

$$X=\frac{{{X}^{4}}+{{s}_{3}}I}{{{s}_{3}}trB+{{t}_{3}}}=\frac{{{B}^{4}}+15 I}{60+\left( -4 \right)}=
\left( \begin{array}{rr}
   56 \,\, &112 \\ 
  56 \,\,&168 \\ 
\end{array} \right):56 = \left( \begin{array}{rr}
   1 \,\,\, & 2 \\ 
  1 \,\, & 3 \\ 
\end{array} \right).$$

Next, using the recursive formula
${{s}_{5}}={{s}_{4}}trB+{{t}_{4}}=trB\cdot 56+\left( -15 \right)=209$, where ${{t}_{4}}=-15$, ${{s}_{4}}=56$. 

Provided that ${{B}^{4}}=\left( \begin{array}{rr}
   41 \,\, & 112 \\ 
  56 \,\, & 153 \\ 
\end{array} \right)$ and ${{B}^{5}}=\left( \begin{array}{cc}
    153 &  418\\
    209 & 571
\end{array} \right)$, 
then ${{s}_{4}}E=\left( \begin{matrix}
   56 & 0  \\
   0 & 56  \\
\end{matrix} \right)$ and $X=\frac{{{X}^{5}}+{{s}_{4}}E}{{{s}_{4}}trB+{{t}_{4}}}=\frac{{{B}^{4}}+56E}{trB\cdot 56+\left( -15 \right)}=\left( \begin{array}{rr}
   209\,\,\, & 418 \\ 
  209\,\,\, & 627 \\ 
\end{array} \right):209$ 
  therefore  $\sqrt[5]{A}=\left( \begin{array}{rr}
   1 \,\,\,\,\,\, & \frac{418}{209} \\ 
  \frac{209}{209}\,\, &  \frac{627}{209} \\ 
\end{array} \right)=\left( \begin{array}{rr}
   1\,\,\, & 2 \\ 
  1\,\,\, & 3 \\ 
\end{array} \right)$. 
Which fully confirms the formula under the conditions of this sequence.

\section{Analytical formula of square root in $SL_3{(\mathbb{Z}})$.}
Let $P={{\lambda }_{1}}{{\lambda }_{2}}+{{\lambda }_{1}}{{\lambda }_{3}}+{{\lambda }_{2}}{{\lambda }_{3}}$, then $\tr^2{(B)}- \tr{({B}^{2})}=2P$. Consider the matrix  equation in 
$SL_3{(\mathbb{Z}})$                 
\[~{{X}^{2}}=A. \]
If according to the theorem \ref{criterionSL(Z)} a solution lays in $SL_3{(\mathbb{Z}})$ then solution of previous equation is a matrix $B$ defined by the formula:
\begin{equation}\label{root3on3}
  B=\frac{{{A}^{2}}-\frac{\tr^2\left( B \right)+\tr{{\left( A \right)}}}{2}A- \tr\left( B \right)}{\left( 1-\frac{\tr^2 \left( B \right) - \tr{{\left( A \right)}}}{2}\tr B\right)}.  
\end{equation}


The characteristic polynomial  $-{{X}^{3}}+\tr\left( B \right){{X}^{2}}-PX+E=0$. 
Applying  Celly Hamilton Theorem for $B$ we have
$-{{B}^{3}}+\tr\left( B \right){{B}^{2}}-PB+E=0$ thence
${{B}^{3}}=\tr\left( B \right){{B}^{2}}-PB+E$.
Multiplying last equation on $B$ admit us obtain 
${{B}^{4}}=\tr\left( B \right){{B}^{3}}-P{{B}^{2}}+B$.

Left side of this equation can be brought into the form $tr\left( B \right)\left( tr\left( B \right){{B}^{2}}-PB+E \right)-P{{B}^{2}}+B={{B}^{2}}\left( tr{{\left( B \right)}^{2}}-P \right)+B\left( 1- P\tr\left( B \right) \right)+tr\left( B \right)$.
Taking into account the relation between  spectral invariant \cite{Klyach} that is bellow 
\[\tr^{2}{{\left( B \right)}}-P=\tr{{\left( B \right)}^{2}}-\frac{\tr{{\left( B \right)}^{2}}-\tr\left( {{B}^{2}} \right)}{2}=\frac{\tr{{\left( B \right)}^{2}}+\tr\left( {{B}^{2}} \right)}{2}=\frac{\tr{{\left( B \right)}^{2}}+\tr\left( A \right)}{2}.\]

Additionally we obtain $$ P= \frac{\tr^2 \left( B \right) - \tr{{\left( A \right)}}}{2}.$$

 Thus we can state relation between spectral invariants: 
$\tr{{\left( A \right)}^{2}}-P=\frac{\tr{{\left( B \right)}^{2}}+\tr\left( A \right)}{2}$.
Substituting this as coefficient of second power of matrix
${{B}^{4}}={{B}^{2}}\left(\tr{{\left( B \right)}^{2}}-P \right)+B\left( 1- P \tr\left( B \right) \right)+\tr\left( B \right)=\left( \frac{\tr{{\left( B \right)}^{2}}+\tr\left( {{B}^{2}} \right)}{2} \right){{B}^{2}}+\left( 1-\operatorname{P}\tr\left( B \right) \right)B+\tr\left( B \right)$.
It yields the  equality ${{B}^{4}}-\left( \frac{\tr{{\left( B \right)}^{2}}+\tr\left( {{B}^{2}} \right)}{2} \right){{B}^{2}}-\tr\left( B \right)=\left( 1-\operatorname{P}tr\left( B \right) \right)B$. 
Taking into account that $B^4=A^2$ and $B^2=A$, the last equality entails the formula of the root:
$B=\frac{{{A}^{2}}-\frac{tr\left( {{B}^{2}} \right)+tr^{2}{{\left( B \right)}}}{2}A-tr\left( B \right)}{\left( 1-\operatorname{P}tr\left( B \right) \right)}=\frac{{{A}^{2}}-\frac{tr\left( {{B}^{2}} \right)+tr^{2}{{\left( B \right)}}}{2}A-tr\left( B \right)}{\left( 1-\frac{tr^{2}\left( {{B}} \right)-tr{{\left(A \right)}}}{2}trB \right)}=\frac{{{A}^{2}}-\frac{tr\left( A \right)+tr^{2}{{\left( B \right)}}}{2}A-tr\left( B \right)}{\left( 1-\frac{tr^2\left( B \right)-tr{{\left( A \right)}}}{2}trB \right)}$.
Thus, we obtain the root formula
$$B=\frac{{{A}^{2}}-\frac{tr^2\left( B \right)+tr{{\left( A \right)}}}{2}A-tr\left( B \right)}{\left( 1-\frac{tr^2 \left( B \right) - tr{{\left( A \right)}}}{2}trB \right)}.$$

\begin{exm}
Let $B=\left( \begin{matrix}
   3 & 0 & 2  \\
   0 & 1 & 0  \\
   1 & 0 & 1  \\
\end{matrix} \right)$ then ${{B}^{2}}=A=\left( \begin{matrix}
   11 & 0 & 8  \\
   0 & 1 & 0  \\
   4 & 0 & 3  \\
\end{matrix} \right)$,  ${{B}^{4}}={{A}^{2}}=\left( \begin{matrix}
   153 & 0 & 112  \\
   0 & 1 & 0  \\
   56 & 0 & 41  \\
\end{matrix} \right)$. Also
 $trB=5$ and $trA=15$.
    \end{exm}

Let us calculate the values of intermediate quantities
$t{{r}^{2}}\left( {{B}^{2}} \right)=225$, $tr{{B}^{2}}=trA=15$,
$t{{r}^{2}}B-trA={{5}^{2}}-15=10$, $\left( t{{r}^{2}}B-trA \right)trB=\left( {{5}^{2}}-15 \right)5=50$. 
 Now we calculate the coefficient $\delta= {\left( 1-\frac{tr^2 \left( B \right) - tr{{\left( A \right)}}}{2}trB \right)}$ from the denominator of \eqref{root3on3}. For this goal $\frac{\left( \tr {^{2}}B-\tr A \right)}{2}\tr B=\frac{\left( {{5}^{2}}-15 \right)}{2}5=25$, $1-\frac{\left( \tr{^{2}} B-\tr A \right)}{2}\tr B=1-\frac{\left( {{5}^{2}}-15 \right)}{2}5=1-25=-24$ that is $\delta$ in the denominator.

The coefficient ${{\alpha }_{}}=\frac{t{{r}^{2}}B+trA}{2}$ in the numerator is equal
${{\alpha }_{}}=\frac{t{{r}^{2}}B+trA}{2}=20$ because of $t{{r}^{2}}B+trA=25+15=40$, thence $\alpha A=\left( \begin{matrix}
   220 & 0 & 160  \\
   0 & 20 & 0  \\
   80 & 0 & 60  \\
\end{matrix} \right)$.

To find a matrix in the nominator we calculate auxiliary matrix polynomial: 
\[{{A}^{2}}-\alpha A=\left( \begin{matrix}
   153 & 0 & 112  \\
   0 & 1 & 0  \\
   56 & 0 & 41  \\
\end{matrix} \right)-\left( \begin{matrix}
   220 & 0 & 160  \\
   0 & 20 & 0  \\
   80 & 0 & 60  \\
\end{matrix} \right)=-\left( \begin{matrix}
   67 & 0 & 48  \\
   0 & 19 & 0  \\
   24 & 0 & 19  \\
\end{matrix} \right)\]

So the matrix in nominator:
 ${{A}^{2}}-\alpha A-trA\cdot E=-\left( \begin{matrix}
   67 & 0 & 48  \\
   0 & 19 & 0  \\
   24 & 0 & 19  \\
\end{matrix} \right)-\left( \begin{matrix}
   5 & 0 & 0  \\
   0 & 5 & 0  \\
   0 & 0 & 5  \\
\end{matrix} \right)=-\left( \begin{matrix}
   72 & 0 & 48  \\
   0 & 24 & 0  \\
   24 & 0 & 24  \\
\end{matrix} \right)$

Dividing on $\delta=-24$ we obtain the final matrix
 $B=\left( \begin{matrix}
   3 & 0 & 2  \\
   0 & 1 & 0  \\
   1 & 0 & 1  \\
\end{matrix} \right).$

\setcounter{equation}{0}
\setcounter{figure}{0}

{\bf Conclusion.}
Size of minimal generated sets of $ESL(n, \mathbb{Z} )$ and $ESL(n, \mathbb{F}_p)$ as involutive as well as not involutive was found  by us in this research.

The Mazurov triples of involutions as the generator systems for $ESL(5, \mathbb{Z} )$ and $ESL(5, \mathbb{F}_p)$ are researched, the minimality of this triples is proved.

 \textbf{Sources of Funding} for Research Presented in a Scientific Article or Scientific Article Itself
This work was partially supported by a grant from the \textbf{Simons Foundation} (\textbf{SFI-PD-Ukraine-00017674, Ruslan Skuratovskii}).

\normalsize

\small
\begin{thebibliography}{99}
\bibitem{SkuESL} \emph{Skuratovskii Ruslan, Lysenko S. O.} Extended Special Linear group $ESL_2(F)$ and matrix equations in $SL_2(F)$, $ESL_2(Z)$ and $GL_2(F_p)$. WSEAS TRANSACTIONS on MATHEMATICS DOI: 10.37394/23206.2024.23.68

\bibitem{Amit} \emph{Amit Kulshrestha and Anupam Singh.} Computing $n$-th roots in $SL_2(Z)$ and Fibonacci polynomials.
{\it Proc. Indian Acad. Sci.} (Math. Sci.) (2020) 130:31 https://doi.org/10.1007/s12044-020-0559-8.

\bibitem{AutofUnimod} L. K. Hua, I. Reiner. Automorphisms of the Unimodular Group. Transactions of the American Mathematical Society.  1951. doi 10.1090/s0002-9947-1951-0043847-x


\bibitem{LatCry}  Ajtai, Miklos  "Generating Hard Instances of Lattice Problems". Proceedings of the Twenty-Eighth Annual ACM Symposium on Theory of Computing. (1996). pp. 99 - 108. CiteSeerX 10.1.1.40.2489. doi:10.1145/237814.237838. ISBN 978-0-89791-785-8. S2CID 6864824.

 \bibitem{NTRU} FOUQUE, Pierre-Alain et al. Falcon: Fast-Fourier Lattice-based Compact Signatures over NTRU. 2020. Available from the Internet on <https://falcon-sign.info/>, accessed in November 8th, 2020.

\bibitem{SkuKeyExchange} R.V., Skuratovskii, O., Hordiienko, Double Threshold Condition for Multisignature in the Blockchain and Key Exchange Protocol based on Non-Commutative Groups having Structure of a Semidirect Product. Wseas Transactions on Systems, Volume 24, 2025., pp. 755-770.

\bibitem{BookLatCry} Guneysu, Tim; Lyubashevsky, Vadim; Poppelmann, Thomas (2012). "Practical Lattice-Based Cryptography: A Signature Scheme for Embedded Systems" (PDF). Cryptographic Hardware and Embedded Systems --- CHES 2012. Lecture Notes in Computer Science. Vol. 7428. IACR. pp. 530–547. doi:10.1007/978-3-642-33027-8 31. 

\bibitem{Un}  \emph{Micheli, G., Schnyder, R.}, The density of unimodular matrices over integrally closed subrings of function fields, Contemporary Developments in Finite Fields and Applications, World Scientific, (2016). pp. 244-253

\bibitem{Levch} \emph{Levchuk, D. V.} On generation of the group $PSL_n(Z+iZ)$ by three involutions,
two of which commute / D. V. Levchuk, Ya. N. Nuzhin // Journal SFU. Serie Math-Ph.
 2008. V. 1, Num. 2. pp. 133–139.

 \bibitem{Maz} Mazurov, V. D. The Kourovka notebook: Unsolved Problems in Group Theory
/ Eds. V. D. Mazurov, E. I. Khukhro // Sobolev Institute of Mathematics,
Novosibirsk, 2022, Num. 20.


\bibitem{Klyach} \emph{Klyachko Anton A., Baranov D. V.} Economical adjunction of square roots to groups. Sib. math. journal, Volume 53 (2012), Number 2, pp. 250-257.

 \bibitem{Suds} H. A. Janabi, L. Hethelyi and E. Horvoth (2020)
     {\em Journal of Group Theory.}
     TI subgroups and depth 3-subgroups in simple Suzuki groups.
     https://doi.org/10.1515/jgth-2020-0044


\bibitem{Zu} N. D. Zyulyarkina, “On the commutation graph of cyclic TI-subgroups in linear groups”, // Proc. Steklov Inst. Math. (Suppl.), 279, suppl. 1 (2012), 175–181.

\bibitem{Nuz}
Ya. N. Nuzhin, Generating triples of involutions of groups of Lie type over a finite field of odd characteristic. II, Algebra and Logic, 36:4(1997), 422-440

\bibitem{LinearShpringer}
Jurg Liesen Volker Mehrmann. {\em Linear Algebra. Springer Undergraduate Mathematics Series. Springer International Publishing Switzerland 2015 (2015).} DOI https://doi.org/10.1007/978-3-319

\bibitem{Mersl}
Yu. I. Merzlyakov, Automorphisms of two-dimensional congruence groups, Algebra and Logic, 10.1007/BF02218574, 12, 4, (262-267), (1973).

\bibitem{Humph}
Stephen P. Humphries. Generation of Special Linear Groups by Transvections. Journal
of Algebra 99 (1986), p. 480—495.



\bibitem{VSEM} M.~A.~VSEMIRNOV.
ON (2,3)-GENERATION OF MATRIX GROUPS
OVER THE RING OF INTEGERS. {\it St. Petersburg Math. J.}
 (2008),  Vol. 19 No. 6, 883--910.

\bibitem{Leem} Dimitri Leemans. String C-group representations of almost simple groups: A
survey. {\it Contemporary Mathematics }
Volume 764, 2021 https://doi.org/10.1090/conm/764/15335.


\bibitem{SusUTn}
A.~S.~Oliinyk, V.~I.~Sushchanskii, Free group of infinite unitriangular matrices,
Mat. Zametki, 2000, Volume 67, Issue 3, 382–386.

\bibitem{Gol} \textit{W. Holubowski}, Subgroups of unitriangular matrices, J. Math. Sci. 145 (2007), no. 1, 47734780.

\end{thebibliography}
\end{document}